\documentclass[12pt]{amsart}

\usepackage{DKstyle}

\newcommand{\vertiii}[1]{{\left\vert\kern-0.25ex\left\vert\kern-0.25ex\left\vert #1
    \right\vert\kern-0.25ex\right\vert\kern-0.25ex\right\vert}}
\theoremstyle{plain}

\usepackage{caption}
\usepackage[labelfont=rm]{subcaption}

\usepackage[pagebackref=true, colorlinks]{hyperref}

\hypersetup{pdffitwindow=true,linkcolor=blue,citecolor=blue,urlcolor=blue,menucolor=blue}

\usepackage{comment}

\begin{document}
\title{Shrinking targets for discrete time flows on hyperbolic manifolds}
\author{Dubi Kelmer}

\thanks{Dubi Kelmer is partially supported by NSF grant DMS-1401747.}
\email{kelmer@bc.edu}
\address{Boston College, Boston, MA}


\date{\today}%
\dedicatory{}%
\commby{}%

\begin{abstract}
We prove dynamical Borel Canteli Lemmas for discrete time homogenous flows hitting a sequence of shrinking targets in a hyperbolic manifold. These results apply to both diagonalizable and unipotent flows, and any family of measurable shrinking targets. As a special case, we establish logarithm laws for the first hitting times to shrinking balls and shrinking cusp neighborhoods, refining and improving on perviously known results. 
\end{abstract}

 \maketitle
 \section{Introduction}
 Consider a dynamical system given by the iteration of a measure preserving transformation $T$ on a probability space $(\mathcal{X},\mu)$.  For any sequence, $\{B_m\}_{m\in \N}$, of measurable subsets of $\cX$, the Borel-Cantelli Lemma implies that if $\sum_{m}\mu(B_m)<\infty$ then $\{m: T^m x\in B_m\}$ is finite for a.e. $x\in \cX$. Conversely, assuming pairwise independence,
we also have that if $\sum_{m}\mu(B_m)=\infty$ then  $\{m:T^m x\in B_m\}$ is infinite for a.e. $x$.  
The condition of pairwise independence is too restrictive to hold in most deterministic systems, nevertheless, in many cases one can still obtain a converse statement under some additional regularity conditions on the target sets. Such results, usually referred to as Dynamical Borel-Cantelli Lemmas,
were established in many dynamical systems with fast mixing. For example, the exponential rate of mixing for diagonalizable group actions on homogenous spaces was used in \cite{Sullivan1982, KleinbockMargulis1999,Maucourant06,GorodnikShah11,KleinbockZhao2017} to study various shrinking target problems in this setting, while in \cite{ChernovKleinbock01,Dolgopyat04,Galatolo07} similar results were obtained for other dynamical systems on more general metric spaces having exponential (or super polynomial) rate of mixing.  We note that in all these cases, in addition to fast mixing, some additional geometric assumptions on the shrinking sets were needed (e.g., they are assumed to be cusp neighborhoods or metric balls shrinking to a point). 

For systems with a polynomial rate of mixing the problem is more subtle. In \cite{GalatoloPeterlongo10,HaydnNicolPerssonVaienti13} it was shown that for shrinking metric balls, $B_m$, that don't shrink too fast, a sufficiently fast polynomial rate of mixing implies that $\{m: T^m x\in B_m\}$ is infinite for a.e. $x\in \cX$. On the other hand, in \cite{Fayad06}, Fayad gave examples of a dynamical system with a polynomial rate of mixing, and a sequence of shrinking metric balls, $\{B_m\}$, with $\sum_m \mu
(B_m)=\infty$ (and even with  $m\mu(B_m)$ unbounded), for which generic orbits eventually miss the shrinking targets. There is thus no hope to prove a general dynamical Borel Cantelli lemma for systems with polynomial mixing rate.

In this paper we consider the problem for the dynamical system given by a discrete time flow on a hyperbolic manifold, $\cM$, or rather its frame bundle $\cX$. This includes the geodesic flow, but also unipotent flows having a polynomial rate of mixing which can be arbitrarily slow. Nevertheless, by using new ideas from \cite{GhoshKelmer15} utilizing an effective mean ergodic theorem to study shrinking target problems, together with results from spectral theory of hyperbolic manifolds, we are able to prove a dynamical Borel-Cantelli Lemma for these flows for any family of shrinking targets in $\cM$. This seems to be the first result of this kind that applies to any family of shrinking targets with no assumptions on their regularity, and also the first result that applies to dynamical systems with arbitrarily slow polynomial mixing

\subsection{The general setup}
Let $\bH^n$ denote real hyperbolic $n$-space and $G=\Isom(\bH^n)$ its group of isometries (so $G\cong\SO_0(n,1)$, and we can identify $\bH^{n}=G/K$ with $K\cong\SO(n)$ a maximal compact group). Here and below we will always use $n$ to denote the real dimension of $\bH^n$.
For $\G\leq G$ a lattice,  let $\cM=\G\bk \bH^n$ denote the corresponding hyperbolic manifold (or orbifold) and let $\cX=\G\bk G$
(that we can identify as the frame bundle of $\cM$). The homogenous space $\cX$ has a natural $G$-invariant probability measure, $\mu$, coming from the Haar measure of $G$, and the projection of this measure to $\cM$, still denoted by $\mu$, is the hyperbolic volume measure.

A discrete one parameter group is a subgroup, $\{g_m\}_{m\in \Z}\leq  G$, satisfying that $g_mg_{m'}=g_{m+m'}$ for any $m,m'\in \Z$. Given an unbounded discrete one parameter group, its right action on $\cX=\G\bk G$ generates a measure preserving discrete time homogenous flow, which is ergodic by Moore's ergodicity theorem.

A sequence of sets, $\cB=\{B_m\}_{m\in \N}$, is called a family of shrinking targets if $B_{m+1}\subseteq B_m$ for all $m$ and $\mu(B_m)\to 0$. We say that a set $B\subseteq \cX$ is spherical if it is invariant under the right action of $K$, and note that any set, $\tilde B\subseteq \cM$ in the base manifold, can be lifted to a spherical set $B\subseteq \cX$.
We will thus identify any family of shrinking targets in $\cM$ with a corresponding family of spherical shrinking targets in $\cX$. 

\subsection{Dynamical Borel-Cantelli }
Let $G=\Isom(\bH^n)$, $\G\leq G$ a lattice, and let $\cM=\G\bk \bH^n$ and $\cX=\G\bk G$ be as above. 
Consider the dynamical system given by a discrete time homogenous flow, $\{g_m\}_{m\in \Z}$ acting on $\cX$.
Given a sequence $\cB=\{B_m\}_{m\in \N}$ of shrinking targets we say that the orbit of a point $x\in \cX$ is hitting the targets if the set $\{m: xg_m\in B_m\}$ is unbounded and denote the set of points with hitting orbits by $\cA_h(\cB)$. The following result gives precise conditions on when $\cA_h(\cB)$ is a null set and when it is of full measure.
\begin{thm}\label{t:BCT}
Assume that $n=\dim_\R(\cM)\geq3$.
For any family, $\cB$, of spherical shrinking targets we have that $\cA_h(\cB)$ is of full measure (respectively a null set) if and only if the series 
$\sum_m \mu(B_m)$ diverges (respectively converges).
Moreover, if the sequence $\{m\mu(B_m)\}_{m\in \N}$ is unbounded, then there is a subsequence $m_j$ such that for a.e. $x\in \cX$ 
\begin{equation}\label{e:asymp}\lim_{j\to\infty}\frac{\#\{m \leq m_j : xg_m\in B_{m_j}\}} {m_j\mu(B_{m_j})}=1.\end{equation}
\end{thm}
\begin{rem}
We note that any unbounded one parameter subgroup in $G$ is either diagonalizable or unipotent.
If the group $\{g_m\}_{m\in \Z}$ is diagonalizable the same results holds also for $n=2$.
For unipotent flows, when $n=2$, the limit \eqref{e:asymp} still holds under the stronger assumption that there is some $\eta>0$ so that  $\{m^{1-\eta}\mu(B_m)\}_{m\in \N}$ is unbounded; or if $B_1$ is pre-compact and the sequence $\{\frac{m\mu(B_m)}{(\log m)^2}\}_{m>1}$ is unbounded. \end{rem}
\begin{rem}
It is remarkable that this result holds for any family of spherical shrinking targets with no additional assumption on their geometry or regularity, and applies both for diagonalizable and unipotent flows. Moreover, this result is new even for the classical setting when the targets are shrinking metric balls and the flow is the discrete time geodesic flow, where the best previously known result was  \cite{KleinbockZhao2017}, obtained the same conclusion only under an additional assumption on the rate of decay.
\end{rem}

%



\begin{rem}
Given a sequence $\cB$ of shrinking targets it is not hard to see that $x\in \cA_{\rm h}(\cB)$  implies that there is a subsequence $m_j$ such that $\{m \leq m_j : xg_m\in B_{m_j}\}\neq \emptyset$ for all $j$.  When the sequence $\{m\mu(B_m)\}_{m\in \N}$ is unbounded, \eqref{e:asymp} shows that there is a subsequence so that these intersections are not only non-empty, but in fact contain asymptotically the expected number of elements. When $\{m\mu(B_m)\}_{m\in \N}$ is bounded, the asymptotics \eqref{e:asymp} can not be expected to hold (e.g. if $m\mu(B_m)$ is bounded away from the integers).
Nevertheless, in this case we still have that for a.e. $x\in \cX$
\begin{equation}\label{e:SBC}
\lim_{M\to\infty}\frac{\#\{m\leq M: xg_m\in B_m\}}{\sum_{m\leq M} \mu(B_m)}=1.
\end{equation}
\end{rem}

\subsection{Orbits eventually always hitting}
Next we study the subtler point of whether the finite orbits $\{xg_k: k\leq m\}$ hit or miss the targets $B_m$.
We say that an orbit of a point $x\in \cX$ is {\em{eventually always hitting}}  if $\{xg_k:k\leq m\}\cap B_m\neq\emptyset$ for all sufficiently large $m$, and denote by $\cA_{\rm ah}(\cB)$ the set of points with such orbits.\footnote{This notion seems closely related to the notion of $\psi$-Dirichlet numbers introduced in \cite{KleinbockWadleigh2016} in the context of Diophantine approximations.}
This gives rise to the following natural question: Under what conditions on the shrinking rate of $\mu(B_m)$, can one deduce that $\cA_{\rm ah}(\cB)$ has full (or zero) measure? 

%

A soft general argument (see Proposition \ref{p:converse} below) shows that if there is $c<1$ such that the set 
$\{m: m\mu(B_m)\leq c\}$ is unbounded then $\cA_{\rm ah}(\cB)$ is a null set.  
It is not hard to construct a sequence of sets with $m\mu(B_m)$  unbounded, having a subsequence with, say, $m_j\mu(B_{m_j})\leq 1/2$. Hence in order to conclude that a generic orbit is eventually always hitting we need a stronger assumption on the rate of shrinking. Our next result shows that the summability condition
\begin{equation}\label{e:summable}
\sum_{j=0}^\infty \frac{1}{2^{j}\mu(B_{2^j})}<\infty,
\end{equation}
is sufficient to show that $\cA_{\rm ah}(\cB)$ has full measure. Moreover, if we further assume that 
\begin{equation}\label{e:double}
 \mu(B_{2^{j}})\ll \mu(B_{2^{j+1}}) 
\end{equation}
we can even show that, generically, the intersections  have roughly the expected number of elements. 

Here and below we use the notation $A(t)\ll B(t)$ or $A(t)=O(B(t))$ to indicate that there is a constant $c>0$ such that $A(t)\leq cB(t)$. We use subscripts to indicate that constant depends on
additional parameters. We also write $A(t)\asymp B(t)$ to indicate that
$A(t)\ll B(t)\ll A(t)$. 
With these notations we have

\begin{thm}\label{t:asymp}
Retain the notation of Theorem \ref{t:BCT}. Then, \eqref{e:summable} implies that $\cA_{\rm ah}(\cB)$ is of full measure and  \eqref{e:summable} and \eqref{e:double} imply that for a.e. $x\in \cX$ for all sufficiently large $m$ 
$$\#\{k \leq m : xg_k\in B_{m}\}\asymp  m\mu(B_m)$$
\end{thm}
\begin{rem}
When $\{g_m\}_{m\in \Z}$ is diagnolizable the same result holds also for $n=2$. For unipotent flows, when  $n=2$, the same conclusion holds if we replace the assumption \eqref{e:summable} by the stronger assumption that $m \mu(B_m)\gg m^\eta$ for some $\eta>0$, or by the assumption that $B_1$ is pre-compact and 
$\sum_{j=0}^\infty \frac{j^2}{2^{j}\mu(B_{2^j})}<\infty.$
\end{rem}

In general, we don't know if there is a clean zero/one law for the measure of $\cA_{\rm ah}(\cB)$ as we have for $\cA_{\rm h}(\cB)$. However, if the shrinking sets decay polynomially, in the sense that $\mu(B_m)\asymp m^{-\eta}$ for some $\eta>0$, then \eqref{e:double} automatically holds and \eqref{e:summable} holds when $\eta<1$. We thus get the following clean result:
\begin{cor}\label{c:regular}
For $G=\Isom(\bH^n)$ with $n\geq 2$, let  $\{g_m\}_{m\in\Z}\leq G$ denote an unbounded discrete one parameter group and $\{B_m\}_{m\in \N}$ a family of spherical shrinking targets in $\cX=\G\bk G$ with $\mu(B_m)\asymp m^{-\eta}$.  If $\eta< 1$ (respectively $\eta>1$) then for a.e. $x\in \cX$ for all $m$ sufficiently large $\{k \leq m : xg_k\in B_{m}\}\neq \emptyset$ (respectively  $\{k \leq m : xg_k\in B_{m}\}= \emptyset$). Moreover, when $\eta<1$ for a.e. $x\in \cX$, for all $m$ sufficiently large 
$$\#\{k \leq m : xg_k\in B_{m}\}\asymp  m^{1-\eta}.$$
\end{cor}

\subsection{Non-spherical shrinking targets}
%

%
When the group $\{g_m\}_{m\in \Z}$ is diagonalizable, taking additional advantage of the exponential rate of mixing, we can adapt our method to get similar results also for non-spherical shrinking targets that are not too irregular.

To make the notion of regular more precise, given parameters $c,\alpha>0$ and a Sobolev norm $\cS$ on $\cX=\G\bk G$ (see section \ref{s:norms} blow) we say that a set $B\subseteq \G\bk G$ is $(c,\alpha)$-regular for $\cS$, if there are smooth functions $f^\pm$ approximating the indicator function, $\chi_B$, in the sense that $0\leq f^-\leq \chi_B\leq f^+\leq 1$ with $c^{-1}\int f^+d \mu\leq \mu(B)\leq c\int f^-d\mu$ and $\cS(f^\pm)\leq c\mu(B)^{-\alpha}$. We say that a collection of sets $\cB$ is regular (for $\cS$) if there is some $c>1$ and $\alpha\geq 0$  such that all sets $B\in \cB$ are $(c,\alpha)$-regular. 

\begin{rem}
Without the restriction on the Sobolev norm it is always possible to approximate indicator functions by smooth functions in this way. The additional restriction on the size of the Sobolev norm is rather mild as we allow it to grow as any power of $1/\mu(B)$. In particular, it is not hard to see that, for any Riemannian metric on $G$, the collection of all metric balls and their complements is a regular collection.
\end{rem}
\begin{thm}\label{t:diag}
Let $\{g_m\}_{m\in \Z}\leq G$ be a diagonalizable one parameter group. Let $\cB=\{B_m\}_{m\in \N}$ be a sequence of regular shrinking targets. Then, the condition 
\begin{equation}\label{e:summable2}\sum_{j=1}^\infty \frac{|\log(\mu(B_{2^j}))|}{2^j\mu(B_{2^j})}<\infty\end{equation}
implies that $\cA_{\rm ah}(\cB)$ is of full measure 
and further assuming \eqref{e:double} we get that for a.e. $x\in \cX$ for all sufficiently large $m$
$$\#\{k \leq m : xg_k\in B_{m}\}\asymp  m\mu(B_m)$$
\end{thm}


\subsection{Logarithm laws, penetration depth, and first hitting times}
Our results on shrinking targets have an immediate application for establishing logarithm laws for penetration depth and first hitting times. We first recall some of the known results for these problems.

Let $d(\cdot,\cdot)$ denote the hyperbolic distance on $\cM=\G\bk \bH^n$ that we lift to a $K$-invariant distance function on $\G\bk G$ \footnote{Instead of the hyperbolic distance we can use any $K$-invariant distance function for which the measure of shrinking balls and cusp neighborhoods have similar asymptotics}. With this distance function we consider shrinking targets that are either shrinking balls, 
$$B_{r}(x_0)=\{x:  d(x,x_0)<r\},$$
or (when $\G\bk G$ is non-compact) shrinking cusp neighborhoods
$$B_{r}(\infty)=\{x: d(x,x_0)\geq r\}.$$
We then have that $\mu(B_{r}(x_0))\asymp r^{n}$ for $r<1$ and that  $\mu(B_{r}(\infty))\asymp e^{(1-n)r}$ for $r>1$.

Given a one parameter subgroup $\{g_m\}_{m\in \Z}\leq G$ acting on $\G\bk G$, we define the following quantities, measuring how deep a finite orbit penetrates into these shrinking targets:
$$d_m(x,x_0)=\min_{j\leq m}d(x g_j,x_0), \quad d_m(x,\infty)=\max_{j\leq m}d(xg_j,x_0).$$
When $\{g_m\}_{m\in \Z}$ is diagonalizable,  the results of \cite{Sullivan1982,KleinbockMargulis1999} for the cusp excursions and \cite{Galatolo07,KleinbockZhao2017} for metric balls give the following logarithm laws: 

\begin{equation}\label{e:loglawcusp}\limsup_{m\to\infty}\frac{d_m(x,\infty)}{\log(m)}=\frac{1}{n-1}, \mbox{ for a.e. $x\in \cX$}\end{equation}
and
\begin{equation}\label{e:loglawball}
\limsup_{m\to\infty}\frac{-\log(d_m(x,x_0))}{\log(m)}=\frac{1}{n},  \mbox{ for a.e. $x\in \cX$}.
\end{equation}
When $\{g_m\}_{m\in\N}$ is unipotent,  the logarithm law for cusp excursions,  \eqref{e:loglawcusp}, was established, using different methods, in \cite{Athreya2013} for hyperbolic surfaces, in \cite{KelmerMohammadi12} for hyperbolic $3$-manifolds, and in \cite{Yu16} for hyperbolic manifolds in any dimension (see also \cite{AthreyaMargulis09} for similar results for unipotent flows on the space of unimodular lattices).  To the best of our knowledge there are no similar known results on logarithm laws for unipotent flows penetrating shrinking balls. 

\begin{rem}
The results mentioned above establishing logarithm laws for cusp excursions actually consider  $\limsup_{m\to\infty}\frac{d(xg_m,x_0)}{\log(m)}$ rather than $\limsup_{m\to\infty} \frac{d_m(x,\infty)}{\log(m)}$. However, it is not hard to see that these two quantities are always the same. For the shrinking balls the two corresponding limits, $\limsup_{m\to\infty}\frac{-\log(d_m(x,x_0))}{\log(m)}$ and $\limsup_{m\to\infty}\frac{-\log(d(xg_m,x_0))}{\log(m)}$,  are also the same as long as $x_0$ is not in the orbit of $x$. 
\end{rem}

\begin{rem}
For the cusp neighborhoods, since $d(xg_m, x_0)=d(xg_t,x_0)+O(1)$ for all $t\in (m-1,m+1)$, it is possible to replace the discrete iteration by a continuous one to get that $
\limsup_{t\to\infty}\frac{d(xu_t,x_0)}{\log(t)}=\frac{1}{n-1}$.
For the shrinking balls, however, there is a fundamental difference between continuous and discrete flows. In fact,  \cite{Maucourant06} showed that for the continuous geodesic flow
$\limsup_{t\to\infty}\frac{-\log(d(xu_t,x_0))}{\log(t)}= \frac{1}{n-1},$
is strictly larger than the corresponding limit for the discrete flow. \end{rem}

Applying Theorem \ref{t:BCT} to these families of shrinking targets gives another proof for \eqref{e:loglawcusp} and establishes  \eqref{e:loglawball} also for unipotent flows. Moreover, if we apply corollary \ref{c:regular}  
(noting that the family of shrinking balls and cusp neighborhoods are spherical) we can replace the limit superior by actual limits obtaining the following strong logarithm laws: 
\begin{thm}\label{t:strongloglaw}
For any unbounded one parameter subgroup $\{g_m\}_{m\in \Z}$, we have the following strong logarithm laws: For any $x_0\in \cX$ for a.e. $x\in \cX$
\begin{equation}\label{e:strongloglaw}
\lim_{m\to\infty}\frac{d_m(x,\infty)}{\log(m)}=\frac{1}{n-1},\quad \lim_{m\to\infty}\frac{-\log(d_m(x,x_0))}{\log(m)}=\frac{1}{n},
\end{equation}
and 
for any $c<1$ we have that for a.e. $x\in \cX$ for all sufficiently large $t$,
\begin{equation*}
\#\{m\leq t: d(xg_m,x_0)\geq \tfrac{c\log(t)}{n-1}\}\asymp t^{1-c},\mbox{ and }\#\{m\leq t: d(xg_m,x_0)< \tfrac{1}{t^{\frac{c}{n}}}\}\asymp t^{1-c} 
\end{equation*}
\end{thm}
\begin{rem}
When the flow is diagonalizable we can use Theorem \ref{t:diag} to get the same result also when the distance function is not $K$-invariant (noting that metric balls and their complements are regular).
\end{rem}
\begin{rem}
For the cusp neighborhoods, again using the continuity of the distance function,
we can replace the discrete iteration by a continuous flow and get that 
\begin{equation}\label{e:loglawcont}
\lim_{t\to\infty}\frac{\sup\{d(xg_s,x_0): s\leq t\}}{\log(t)}=\frac{1}{n-1}.
\end{equation}
For shrinking balls, however,  the same argument as in \cite{Maucourant06} will give the upper bound
$$\limsup_{t\to\infty}\frac{\sup\{-\log(d(xg_s,x_0)): s\leq t\}}{\log(t)}\leq \frac{1}{n-1},$$
and a simple adaptation of our method for continuous flows will only give the (strictly smaller) lower bound
$$\frac{1}{n}\leq \liminf_{t\to\infty}\frac{\sup\{-\log(d(xg_s,x_0)): s\leq t\}}{\log(t)}.$$
We suspect that the correct lower bound should also be $\tfrac{1}{n-1}$, but we are not able to show this using our current method.
\end{rem}


The logarithm laws given in \eqref{e:strongloglaw} are closely related to logarithm laws for first hitting times. The first hitting time of the orbit $\{xg_m\}_{m\in \N}$  to a small neighborhood of $x_0\in \cX\cup \{\infty\}$ is given by 
$$\tau_r(x;x_0)=\min\{ m\in \N: xg_m\in B_r(x_0)\}.$$
As a direct consequence of Theorem \ref{t:strongloglaw} 
we have 
\begin{cor}\label{c:hitting}
For any one parameter subgroup $\{g_m\}_{m\in \Z}$, for any $x_0\in \cX$ and a.e. $x\in \cX$ the first hitting time function satisfies
$$\lim_{r\to \infty} \frac{\log(\tau_r(x;\infty))}{r}=n-1 \mbox{ and }
\lim_{r\to 0} \frac{\log(\tau_r(x;x_0))}{-\log r}=n.$$
\end{cor}

\subsection{Multiparameter actions}
The maximal unipotent subgroup of $G=\SO_0(n,1)$ is isomorphic to $\R^{n-1}$. Instead of a one parameter flows we can also consider multi-parameter actions, given by the action of a discrete subgroup $H$ of the unipotent group (that is isomorphic to $\Z^d$ for some $1\leq d\leq n-1$). We will actually state and prove most of our results for unipotent actions in this generality. The restriction of these results to the case of $d=1$ will then give the above results for one parameter unipotent flows.

\subsection{Other homogenous spaces}
The methods introduced here to attack shrinking target problems for discrete time flows on hyperbolic manifolds can be adapted to more general homogenous space. In fact, in forthcoming joint work with Yu \cite{KelmerYu17} we use these methods to obtain similar results for one parameter flows on homogenous spaces $\G\bk G$ with $G$ any non-compact simple group, and $\G\leq G$ a lattice. When $G$ is of rank one the arguments are very similar to the ones used here, while for groups of higher rank, the detailed spectral analysis used here is replaced with more robust arguments using an effective version of property $(T)$.

\subsection*{Acknowledgments} 
We thank Peter Sarnak and Dimitry Kleinbock for helpful discussions and Shucheng Yu for his comments. 

\section{Preliminaries}

\subsection{Coordinates}\label{s:coordinates}

Let $G=\SO_0(n,1)$ and $K\leq G$ a maximal compact subgroup (so $K\cong \SO(n)$).
Let $G=NAK$ be an Iwasawa decomposition of $G$, with $A$ a maximal
diagonalizable subgroup, and $N$ a maximal unipotent subgroup.
Let $\fa$ denote the Lie algebra of $A$ and let $\fa_\C^*$
denote its complexified dual (one dimensional since $G$ is of rank one).
Fix a positive Weyl chamber, denote by $\rho$ the half sum of the positive
roots, and fix a positive $X_0\in \fa$ of norm one with respect to the
Killing form.  We now identify $\fa_\C^*$ with $\C$ via their values
at $X_0$ and, slightly abusing notation, write $\rho=\rho(X_0)\in \R$ so that
$\rho=\frac{n-1}{2}$.

It is sometimes convenient to work with explicit coordinates, and we realize $G$ as the subgroup of $\SL_{n+1}(\R)$ of all matrices preserving the form $x_1^2+\ldots+x_n^2-x_{n+1}^2$. With this realization a general element of $A$ is of the form 
\begin{equation}\label{e:Cartan}a_t=\left(\begin{matrix} I_{n-1} &0 &0\\ 
0& \cosh(t) & \sinh(t)\\  0& \sinh(t) & \cosh(t)\end{matrix}\right),\quad t\in \R\end{equation}
an element of $N$ is of the form 
\begin{equation}\label{e:unipotent}
n_x=\left(\begin{matrix} I_{n-1} & x & x\\ -x^t & 1-\tfrac{\|x\|^2}{2}& \tfrac{\|x\|^2}{2}\\ x^t & -\tfrac{\|x\|^2}{2} & 1+\tfrac{\|x\|^2}{2}\end{matrix}\right),\quad x\in \R^{n-1}\end{equation}
and an element of $K$ is of the form 
$k=\begin{pmatrix} k' & 0\\ 0 &1\end{pmatrix}$ with  $k'\in \SO(n).$

The  Iwasawa decomposition $G=NAK$ gives us coordinates $g=n_xa_tk$, and the Haar measure in these coordinates is given by
$$dg=e^{-2\rho t}dxdtdk,$$
where $dk$ is the probability Haar measure on $K$.
The Cartan decomposition $G=KA^+K$ with $A^+=\{a_t: t\geq 0\}$ gives us another parametrization $g=k_1a_{t}k_2$.
Here $k_1,k_2$ are not uniquely defined but $a_t$ is and the Haar measure of $G$ is given in these coordinates by
$$dg=(2\sinh(t))^{2\rho}dk_1dk_2dt.$$

%
\subsection{Sobolev Norms}\label{s:norms}
Fix a basis $\sB$ for the Lie algebra $\mathfrak{g}$ of $G$. Given a smooth test function $\psi\in C^{\infty}(\G\bk G)$,  define the ``$L^{p}$, order-$d$'' Sobolev norm $\cS_{p,d}(\psi)$  as
$$
\cS_{p,d}(\psi) \ : = \ \sum_{\ord(\sD)\le d}\|\sD\psi\|_{L^{p}(\G\bk G)}
.
$$
Here $\sD$ ranges over monomials in $\sB$ of order at most $d$ and $\sD$ acts on $\psi$ by left differentiation 
(e.g., $X\psi(g)=\frac{d}{dt}(\psi(ge^{tX}))|_{t=0}$). This definition depends on the basis, however, changing the basis $\sB$ only distorts $\cS_{p,d}$ by a bounded factor.

\subsection{Spectral theory}\label{s:spetral}
Denote by $\hat{G}$ the unitary dual of $G$ and by
$\hat{G}^1$ the spherical dual.
We parameterize $\hat{G}^1$ by $\fa_\C^*/W$ where $W$ is the Weyl group.
With this parametrization the tempered representations lie in $i\fa^*_\R$
and the non-tempered representations are contained in $(0,\rho]$.
We use the notation
\[\hat{G}^1=\{\pi_s|s\in i\R^+\cup [0,\rho]\},\]
where the representations $\pi_s, s\in i\R^+$ are the (tempered)
principal series representations, the representations
$\pi_s, s\in (0,\rho)$ are the (non-tempered) complementary series,
and $\pi_\rho$ is the trivial representation. We note that in each representation $\pi_s\in \hat{G}^1$ there is a unique spherical vector up to scaling.

We recall the relation between Laplace eigenvalues and irreducible
representations.  Let $\Omega$ be the Casimir operator of $G$,
which acts on any irreducible representation $V_\pi$ by
scalar multiplication $\Omega v+\lambda(\pi)v=0$.
We may normalize $\Omega$ so that the restriction of $\Omega$ to the space
of right-$K$-invariant functions on $G$ coincides with
the Laplace-Beltrami operator $\triangle$ on $\bH^n=G/K$.
With this normalization we have $\lambda(\pi_s)=\rho^2-s^2$.

\subsection{Spectral decomposition}
For any lattice $\G\leq G$ let $\mu$ denote the $G$ invariant probability measure on $\G\bk G$ and let $\pi$ denote the right regular representation of $G$ on $L^2(\G\bk G)$.
Let $L^2_0(\G\bk G)$ denote the subspace of $L^2(\G\bk G)$ orthogonal to the constant function and let  $L^2_{\mathrm{tmp}}(\G\bk G)\subseteq L^2(\G\bk G)$ be the subspace that weakly contains only tempered representations. We then have a spectral decomposition for any $f\in L^2(\G\bk G/K)$ 
$$f=\langle f,1\rangle+\sum_{k}\langle f,\vf_k\rangle\vf_k +f_0,$$
with $f_0\in L^2_{\mathrm{tmp}}(\G\bk G)$ and $\vf_k\in \pi_{s_k}$ with $s_k\in (0,\rho)$. 
The exceptional forms $\vf_k\in \pi_{s_k}$ above are either cusp forms (vanishing at all cusps), or residual forms (obtained as residues of Eisenstein series).

The Eisenstein series (corresponding to a cusp at infinity) is defined by 
$$E(s,g)=\sum_{\G_\infty\bk \G} e^{-st(\g g)},$$
where $\G_\infty=\G\cap N$ and  $t(g)$ is defined by the relation $g=n_xa_{t(g)}k$.
Similarly, given a different cusp $\fa$, corresponding to a conjugate $N^\fa=\tau_\fa N\tau_\fa^{-1}$ with $\G_\fa=\G\cap N^\fa$ a lattice in $N^\fa$,  the corresponding Eisenstein series is defined as
$$E_\fa(s,g)=\sum_{\G_\fa\bk \G} e^{-st(\tau_\fa\g g)}.$$
These series converge for $\Re(s)>n$ and have meromorphc continuations to the whole complex plane, which are analytic in the half plane $\Re(s)>\rho$, except for a pole at $s=2\rho$ and perhaps finitely many exceptional poles in the interval $[\rho,2\rho)$. 
The residue of the Eisenstein series at an exceptional pole $\sigma\in [\rho,2\rho)$ 
 $$\vf_\sigma(g)=\Res_{s=\sigma}E_\fa(s,g),$$
 is always in $L^2(\G\bk G)$ and generates the complementary representation $\pi_{\sigma-\rho}\in \hat{G}^1$. The fact that residual forms are square integrable can be seen by looking at the Fourier expansion of the Eisenstein series with respect to the different cusps, showing that a residual forms corresponding to a pole at $\sigma\in [\rho,2\rho)$ grows at the cusp at infinity like
 $$\vf_\sigma(na_tk)\asymp e^{(2\rho-\sigma)t},$$ 
 (and similarly at other cusps). From these asymptotics we get the following stronger statement:
 \begin{lem}
 For any $\sigma\in [\rho, 2\rho]$  the corresponding residual form satisfies that $\vf_\sigma\in L^p(\G\bk G)$ for any $p<\frac{2\rho}{2\rho-\sigma}$.
 \end{lem}
 \begin{proof}
 Fixing a fundamental domain for $\G\bk G$ and decomposing it into a compact part and finitely many cusps it is enough to show that $|\vf_\sigma|^p$ is integrable on the cusp neighborhoods. We will show this for the neighborhood of the cusp at infinity of the form  
 $\cF_\infty=\{na_{t} k: n\in \G_\infty \bk N,\; t\geq 1,\; k\in K\},$
 (other cusps are analogous). Since $\G_\infty \bk N$ and $K$ are compact it is enough to show that 
 $\int_{t_0}^\infty |\vf_\sigma(a_t)|^pe^{-2\rho t} dt<\infty$ and using the growth at the cusp this is the same as 
$\int_{t_0}^\infty e^{((2\rho-\sigma)p-2\rho)t}dt<\infty$ which holds for all $p<\frac{2\rho}{2\rho-\sigma}$.

 \end{proof}
 
 \subsection{Decay of matrix coefficients}
For each $s\in i\R^+\cup [0,\rho]$ consider the \emph{spherical function}
defined by
\begin{equation}\label{e:spherical}
\phi_s(g)=\langle \pi_s(g)v,v\rangle
\end{equation}
where $v\in V_{\pi_s}$ is a normalized spherical vector.
The asymptotic behavior of this function is well-known.
For $s\in(0,\rho]$ the function $\phi_s$ is positive and decays like
\[\phi_s(\exp(tX_0))\asymp C_s e^{(s-\rho)t}.\]
For $s\in i\R$, $\phi_s$ is oscillatory and decays like
\[|\phi_{s}(\exp(tX_0))|\ll  t e^{-\rho t}.\]
In particular, applying this for the regular representation on $L^2(\G\bk G)$, recalling our spectral decomposition, we get the following:
\begin{prop}\label{p:smatrixdecay}
For any spherical tempered $\psi,\phi\in L^2_\mathrm{temp}(\G\bk G)$ 
$$|\langle \pi(k_1a_tk_2)\psi,\phi\rangle| \ll t \norm{\psi}\norm{\phi} e^{-\rho t}.$$
While for spherical $\vf\in \pi_{s}$ with $s\in (0,\rho]$ we have 
 $$|\langle \pi(k_1a_tk_2)\vf,\vf\rangle| \ll_s \norm{\vf}^2 e^{(s-\rho) t}.$$
\end{prop}
%
Similar results also hold for non-spherical $K$-finite vectors (see e.g. \cite{Shalom2000}). More generally, by decomposing a smooth vector into its different $K$ types, we can bound matrix coefficients of general smooth vectors. However, in order to insure that the contribution from all $K$-types converges, we need to replace the $L^2$-norms by suitable Sobolev norms. We thus get the following (see e.g., \cite[Proposition 3.9]{MohammadiOh15})
\begin{prop}\label{p:matrixdecay}
There is some $l\geq 1$ (depending only on $G$ and $K$) and $\alpha>0$ depending on the spectral gap of $\G$, such that  for any smooth $\psi,\phi\in L^2_0(\G\bk G)$ 
$$|\langle \pi(k_1a_tk_2)\psi,\phi\rangle| \ll \cS_{2,l}(\psi)\cS_{2,l}(\phi) e^{-\alpha t}.$$
\end{prop}

In order to apply this decay to matrix coefficients of general unipotent and diagonalizable elements  we need to consider their Cartan decomposition given by the following two lemmas. 
\begin{lem}\label{l:kakunipotent}
We have $n_x=k_1a_tk_2$ with $t\asymp 2\log\|x\|_\infty$ uniformly for $x\in \R^{n-1}$ with $\|x\|_\infty\geq 1$.
\end{lem}
\begin{proof}
Consider the Hilbert Schmidt norm $\|g\|=\tr(g^tg)$. On one hand looking at sum of squares of coefficients we get
$\|n_x\|=n+1+4\|x\|_2^2+\|x\|_2^4,$
and on the other hand writing $n_x=k_1a_tk_2$ we see $\|n_x\|=\tr(a_{2t})=n-1+2\cosh(2t)$ so 
$$2\cosh(2t)=2+4\|x\|_2^2+\|x\|_2^4.$$
We have $\|x\|_2\asymp \|x\|_\infty$ and hence $t\asymp 2\log( \norm{x}_\infty)$.
\end{proof}

\begin{lem}\label{l:kakdiag}
For any diagonalizable one parameter group $\{g_m\}_{m\in \Z}$ there is some $c>0$ so that  
$g_m=k_1a_{t_m}k_2$ with $t_m=cm+O(1)$ for all $m\geq 1$.
\end{lem}
\begin{proof}
We can conjugate $g_1=\tau^{-1}\tilde{k}a_c\tau$ with $\tau\in G$ and $\tilde{k}\in K$  commuting with any $a\in A$. 
Decomposing $\tau=an_x k$ we see that $g_m=k^{-1}n_{-x}\tilde{k}^ma_{mc}n_xk$ and taking the Hilbert Schmidt norm we get that 
$$\|g_m\|^2= \|n_{-x}\tilde{k}^ma_{mc}n_x\|^2\asymp_x \|n_{-x}a_{mc}n_x\|^2,$$
where we used that $\|n_{-x}\tilde{k}^mn_x\|_2\asymp 1$ is uniformly bounded and bounded away from zero.
Next, a simple computation using \eqref{e:Cartan} and \eqref{e:unipotent} to compute the sum of squares of entries of $n_{-x}a_{mc}n_x$ gives 
  $ \|n_{-x}a_{mc}n_x\|^2\asymp_x e^{2mc}$. On the other hand, writing $g_m=k_1a_{t_m}k_2$ we have that 
  $\|g_n\|^2=n-1+2\cosh(2t_m)$. Comparing the two expressions we see that  $e^{2t_m}\asymp e^{2mc}$ so that $t_m=cm+O(1)$ as claimed.
\end{proof}

\section{Shrinking targets for $\Z^d$-actions}\label{s:Zdaction}
We now consider the general case of a $\Z^d$-actions of a group $H$ on a probability space $(\cX,\mu)$ and give conditions under which a generic orbit eventually always hits a shrinking target. We will later apply these ideas to the special case of $\cX=\G\bk G$ and $H$ a unipotent or diagonalizable subgroup.

We start by setting some notation. Let $H\cong \Z^d$ be a group acting on a probability space $(\cX,\mu)$, and assume that the action $x\mapsto xh$ is measure preserving and ergodic. 
Fixing an isomorphism, $\iota:\Z^d\to H$, for any $\bbk\in \Z^d$ let $h_\bbk=\iota(\bbk)$ and for any $m\in \N$ consider growing balls $H_m=\{h_\bbk:\, \bbk\in \Z^d,\; \|\bbk\|_\infty\leq m\}$ and forward balls $H_m^+=\{h_\bbk:\, \bbk\in \Z_{>0}^d,\; \|\bbk\|_\infty\leq m\}$. 

Given a family $\cB=\{B_m\}_{m\in \N}$ of shrinking targets in $\cX$, (that is, $B_{m+1}\subseteq B_m$ and $\mu(B_m)\to 0$) we say that the forward orbit of a  point  $x\in \cX$ hits the target if $\{m: xH^+_m\cap B_m\neq \emptyset\}$ is unbounded and that it eventually always hits, if $xH_m^+\cap B_m\neq \emptyset$ for all sufficiently large $m$. We denote by $\cA_{\rm h}(\cB)$ and  $\cA_{\rm ah}(\cB)$ the set of points with hitting orbits and eventually always hitting orbits respectively.  (When considering the special case of a one parameter flow on the space $\cX=\G\bk G$ this definition coincides with the definition of $\cA_{\rm h}$ and $\cA_{\rm ah}$ given in the introduction.)
 We now give some general conditions on a family $\cB$, under which we can conclude that $\cA_{\rm ah}(\cB)$ has full (or zero) measure. 

\subsection{Fast shrinking targets} 
We first show here the rate $\mu(B_m)=\frac{1}{m^d}$ is critical in the sense if a sequence of targets have measure shrinking any faster, then generic orbits are not eventually always hitting. 
\begin{prop}\label{p:converse}
Let $H$ be a measure preserving ergodic $\Z^d$-action on a probability space $(\cX,\mu)$ and let $\cB=\{B_m\}_{m\in \N}$ denote a sequence of shrinking targets. If, along some subsequence we have that $m_j^d\mu(B_{m_j})\leq c<1$, then $\mu(\cA_{\rm ah}(\cB))=0$.
\end{prop}
\begin{proof}
We keep the sequence of shrinking targets $\cB$ fixed and will omit it from the notation denoting by $\cA=\cA_{\rm ah}(\cB)$.
It is not hard to show that the assumption $m_j^d\mu(B_{m_j})\leq c$ implies that $\mu(\cA)\leq c$. Hence if $\cA$ was invariant under the flow, ergodicity would imply that $\mu(\cA)=0$. In general, $\cA$ might not be invariant so we will consider a larger set $\cA\subset \cA^*$ that is invariant and show that $\mu(\cA^*)\leq c$. 

Explicitly, for any integers $\nu,l\geq 0$ let 
$$\cA_{\nu,l}=\bigcap_{m> l}\{x: xH_m^+H_\nu\cap B_m\neq \emptyset\}.$$
so that $\cA=\bigcup_{l\geq 0} \cA_{0,l}$, and we consider the larger set
$$\cA^*= \bigcup_{\nu\geq 0}\bigcup_{l\geq 0}\cA_{\nu,l}.$$

First to show that $\cA^*$ is invariant under $H$ we need to show that $\cA^* h\subseteq \cA^*$ for any $h\in H$.
Let $x\in \cA^* h$ then $xh^{-1}\in \cA^*$ so there are some $l,\nu\geq 0$  with $xh^{-1}\in \cA_{\nu,l}$.
Hence, for all $m\geq l$ we have that $xh^{-1}H_m^{+}H_\nu\cap B_m\neq \emptyset$.  Let $\nu'$ be sufficiently large so that $h^{-1}H_\nu\subseteq H_{\nu'}$ then for all $m\geq l$ we have that $xH_m^+H_{\nu'}\cap B_m\neq \emptyset$ so that $x\in \cA_{l,\nu'}\subseteq \cA^*$.

Next to bound the measure of $\cA^*$ since  $\cA_{\nu,l}\subseteq \cA_{\nu+1,l}$ and $\cA_{\nu,l}\subseteq \cA_{\nu, l+1}$ we have 
$$\mu(\cA^*)=\lim_{\nu\to\infty}\lim_{l\to\infty}\mu(\cA_{\nu,l}),$$
so it is enough to show that $\mu(\cA_{\nu,l})\leq c$ for any $l,\nu\geq 0$. For $l,\nu\geq 0$ fixed, notice that $xH_m^+H_\nu\cap B_m\neq \emptyset$ iff $x\in B_mH_\nu H^-_m$ with $H_m^{-}=\{h^{-1}: h\in H_m^+\}$ so that 
$$\cA_{\nu,l}=\bigcap_{m\geq l}B_mH_\nu H_m^{-},$$
and hence, for any $m\geq l$ 
$$\mu(\cA_{\nu,l})\leq \mu(B_mH_\nu H_m^{-})\leq\mu(B_m)(m+2\nu)^d=m^d\mu(B_m)(1+\frac{2\nu}{m})^d.$$
Let $m_j\to\infty$ with $m_j^d\mu(B_{m_j})\leq c$ then $\mu(\cA_{\nu,l})\leq c(1+\frac{2\nu}{m_j})^d$ and taking $j\to\infty$ we get that $\mu(\cA_{\nu,l})\leq c$. Since this holds for any $l,\nu\geq 0$ we get that $\mu(\cA^*)\leq c$, and since $\cA^*$ is invariant under the (ergodic) $H$-action then $\cA^*$  is a null set and hence so is $\cA$.
\end{proof}

\subsection{Slow shrinking targets}
Next we want to give a condition implying that $\cA_{\rm ah}(\cB)$ is of full measure.  
To do this, for any set $B\subseteq X$ and $m\in \N$ let  
\begin{equation}\label{e:CTBo}
\cC_{m,B}^o=\{x\in \cX: xH_m^+\cap B=\emptyset\},
\end{equation}
denote the set of points with forward $m$-orbit missing a set $B$. 
We show that a sufficiently strong bound on the measure of these sets as $m\to\infty$ will imply that generic orbits will eventually hit all shrinking sets (later we will show how, in some cases, it is possible to bound  $\mu(\cC_{m,B}^o)$ in terms of $m$ and  $\mu(B)$).
\begin{lem}\label{l:CTBo}
Let $\cB=\{B_m\}_{m\in\N}$ denote a sequence of shrinking targets. If there is a subsequence $m_j\to \infty$ with 
$\sum_j \mu(\cC_{m_{j-1},B_{m_{j}}}^o)<\infty$ the $\mu(\cA_{\rm ah}(\cB))=1$.
\end{lem}
\begin{proof}
Let $\cA=\cA_{\rm ah}(\cB)$ and note that $x\not\in \cA$, iff for all $T>0$ there is an integer $m\geq T$ such that $xH_m^+\cap B_m=\emptyset$,
that is, $\cA^c= \bigcap_{T\geq 1}\bigcup_{m>T}\cC_{m,B_m}^o$, with $\cC_{m,B_m}^o$ given in \eqref{e:CTBo}.
 We can separate the union over all $m>T$ into intervals of the form
  $$\bigcup_{m=m_{j-1}}^{m_j}\cC_{m,B_m}^o=\{x\in \cX: \exists m\in [m_{j-1},m_j],\; xH_m^+\cap B_m=\emptyset\}\subseteq \cC_{m_{j-1}, B_{m_j}}^o,$$
 where for the last inclusion note that $xH^+_{m_{j-1}}\cap B_{m_{j}}\subseteq xH_m^+\cap B_m$ for all $m\in [m_{j-1},m_j]$.
We thus get that
$$\cA^c\subseteq \bigcap_{T>1}\bigcup_{m_j>T} \cC_{m_{j-1},B_{m_{j}}}^o,$$
hence $\sum_j \mu(\cC_{m_{j-1},B_{m_{j}}}^o)<\infty$ implies that $\mu(\cA^c)=0$.
\end{proof}
\subsection{Asymptotics}
When $\mu(\cA_{\rm ah}(\cB))=1$ we also want an estimate on the number of elements in the intersections $xH_m^+\cap B_m$ (for typical orbits). To estimate this it is convenient to work with the following unitary averaging operator $\beta_m^+: L^2(\cX)\to L^2(\cX)$ given by
\begin{equation}\label{e:betam}
\beta_m^+(f)(x)=\frac{1}{|H_m^+|}\sum_{h\in H_m^+}f(xh).
\end{equation}
In particular, for $f=\chi_B$, the indicator function of $B$, we have that $\beta_m^+(f)(x)=\frac{|xH_m^+\cap B|}{|H^+_m|}$.
 From ergodicity, for any $f\in L^2(\cX)$ with $\mu(f)=\int fd\mu$, for a.e. $x\in \cX$ we have $\beta^+_m(f)(x)\to \mu(f)$ as $m\to \infty$. We consider the  following set of atypical points  
\begin{equation}\label{e:CTB}
\cC_{m,f}=\{x\in \cX: |\beta^+_m(f)(x)-\mu(f)|\geq \tfrac12\mu(f)\}.
\end{equation}
(in particular, for  $f=\chi_B$ if $x\in \cC_{m,B}^o$ then $\beta^+_m(f)=0$ so $\cC_{m,B}^0\subseteq \cC_{m,f}$).
\begin{lem}\label{l:CTB}
Let $\{f_m\}_{m\in \N}$ be a decreasing sequence of non-negative functions in $L^2(\cX)$ satisfying that  $\mu(f_{2^j})\leq C\mu(f_{2^{j+1}})$ for some $C>1$.
If 
$\sum_j \mu(\cC_{2^{j-1},f_{2^{j}}})<\infty$ then for a.e. $x\in \cX$ for all sufficiently large $m$ 
$$ \beta_m^+(f_m)(x)\geq \frac{\mu(f_m)}{C2^{d+1}},$$ 
and if $\sum_j \mu(\cC_{2^{j+1},f_{2^{j}}})<\infty$ then for a.e. $x\in \cX$ for all sufficiently large $m$ 
$$\beta_m^+(f_m)(x)\leq   C2^{d+1}\mu(f_m).$$
 \end{lem}
\begin{proof}
Let $\cC_\delta$ be the set of all points, $x$, such that for all $T\geq 1$ there is an integer $m>T$ such that $\beta_m^+(f_m)(x)\leq \delta\mu(f_m)$, that is,
 $$\cC_\delta= \bigcap_{T\geq 1}\bigcup_{m>T}\{x:\beta^+_m(f_m)(x)\leq \delta\mu(f_m)\}.$$
 Since the function $m^d\beta^+_m(f_{m'})$ is increasing in $m$ and decreasing in $m'$, we can 
 separate the union into dyadic intervals, and note that for any $j\geq 0$
 $$\big\{x: \exists m\in [2^{j-1},2^j],\; \beta^+_m(f_m)(x)\leq \delta \mu(f_m)\big\}\subseteq \big\{x: \beta^+_{2^{j-1}}(f_{2^j})(x)\leq 2^{d}C\delta \mu(f_{2^j})\},$$
 where we used that $\mu(f_{2^{j-1}})\leq C\mu(f_{2^j})$. In particular, for  $\delta_0= \frac{1}{C2^{d+1}}$ we get that
 $$\bigcup_{m=2^{j-1}}^{2^j}\{x:\beta^+_m(f_m)(x)\leq \delta_0\mu(B_m)\}\subseteq \cC_{2^{j-1},f_{2^j}},$$
 hence,
 $$\cC_{\delta_0}\subseteq \bigcap_{T\geq 1}\bigcup_{j>\log(T)} \cC_{2^{j-1},f_{2^{j}}},$$
and if $\sum_j  \cC_{2^{j-1},B_{2^{j}}}<\infty$ then $\mu(\cC_{\delta_0})=0$. Since for all $x\in \cX\setminus \cC_{\delta_0}$ we have that $\beta^+_m(f_m)(x)\geq  \frac{\mu(f_m)}{C2^{d+1}}$ for all sufficiently large $m$ this proves the first part.
 
 Similarly, if we denote by  $\cC^\delta$  the set of all points such that for all $T\geq 1$ there is an integer $m\geq T$ such that $\beta_m^+(f_m)(x)\geq \delta^{-1}\mu(f_m)$, then
 $$\cC^\delta= \bigcap_{T\geq 1}\bigcup_{m>T}\{x:\beta_m^+(f_m)(x)\geq \delta^{-1}\mu(f_m)\}.$$
This time notice that
 $$\big\{x: \exists m\in [2^{j-1},2^j],\;  \beta_m^+(f_m)(x)\geq \delta^{-1}\mu(f_m)\big\}\subseteq \big\{x: \beta^+_{2^{j}}(f_{2^{j-1}})(x)\geq \frac{ \mu(f_{2^j})}{2^{d}C\delta}\},$$
 and hence, for $\delta_0=\frac{1}{2^{d+1} C}$ as above we have
 $$\bigcup_{m=2^{j-1}}^{2^j}\big\{x:  \beta_m^+(f_m)(x)\geq \delta_0^{-1}\mu(f_m)\big\}\subseteq \cC_{2^j,f_{2^{j-1}}},$$ 
 so that
$$\cC^{\delta_0}\subseteq \bigcap_{T\geq 1}\bigcup_{j>\log(T)} \cC_{2^{j+1},f_{2^{j}}},$$
and the proof follows as above.
 \end{proof}

\section{Effective mean ergodic theorems}
Recall that for an ergodic $\Z^d$-action of a group $H$ on a probability space $(\cX,\mu)$ the mean ergodic theorem states that 
$$\|\beta^+_m(f)-\mu(f)\|_2\to 0,$$
for any $f\in L^2(\cX)$, where $\beta_m^+$ is the unitary averaging operator defined in \eqref{e:betam} and $\mu(f)=\int fd\mu$.  We say that the $H$ action satisfies an effective mean ergodic theorem (with exponent $\kappa>0$) if for all $f\in L^2(\cX)$
\begin{equation}\label{e:MET}\|\beta^+_m(f)-\mu(f)\|_2\ll_\kappa \frac{\|f\|_2}{|H^+_m|^\kappa}.\end{equation}
Using the ideas of \cite{GhoshKelmer15}, by showing that the $H$ action satisfies effective mean ergodic theorems with exponents $\kappa=1/2$ (or arbitrarily close to $1/2$), it is possible to essentially resolve the shrinking target problem for the corresponding $H$ action. Our goal in this section is to prove such bounds for the case where $\cX=\G\bk G$ and the group $H\leq G$ is either a unipotent group (of rank $d$) or a diagonalizable group (of rank one).

\subsection{Unipotent actions}
We first consider the case of unipotent actions. Let $H\leq G$ be a discrete unipotent of rank $d$,  for some $1\leq d< n$. After perhaps conjugating by an element of $K$, we may assume that $H\leq N$ and, without loss of generality, we may assume that 
\begin{equation}\label{e:H}
H=\{n_\bbk: \bbk\in \Z^d\},
\end{equation}
where we identify $\Z^d$ as the first $d$ coordinates of $\Z^{n-1}$. With this convention we get that the truncated orbits and forward orbits are given explicitly by 
\begin{equation}\label{e:Ht}
H_m=\{n_\bbk\in H:  \|\bbk\|\leq m\}.
\end{equation}
and
\begin{equation}
\label{e:Ht+}H_m^+=\{n_\bbk\in H: \bbk\in \Z_{>0}^d,\; \|\bbk\|\leq m\}.
\end{equation}

%


If we attempt to prove \eqref{e:MET} using the decay of matrix coefficients (given in Proposition \ref{p:matrixdecay})
we need to replace $\|f\|_2$ on the right by some Sobolev norm, which will be too costly in the application. Instead, restricting our attention to spherical functions, we can use Proposition \ref{p:smatrixdecay} to show that  \eqref{e:MET} is satisfied with some $\kappa>0$ depending on the spectral gap. Explicitly, if the spectrum of $L^2(\G\bk G)$ contains complementary series $\pi _{s}$ with $s\in(\rho-\frac{d}{2},\rho)$ then we get that $\kappa<\frac{\rho-s}{d}$ is bounded away from $1/2$, which is again an obstacle for solving the shrinking target problem. Instead, we prove the the following modified version of \eqref{e:MET}, treating the contribution of the exceptional spectrum separately. 
\begin{thm}\label{t:SMET}
For any spherical $f\in L^2(\G\bk G)$ we have, 
$$\|\beta^+_mf-\mu(f)\|_2 \ll  \frac{\|f\|_2\log(m)}{|H^+_m|^{1/2}}+ \sum_{s_k\in(\rho-\frac{d}{2},\rho)}\frac{ |\langle f,\vf_k\rangle|}{|H^+_m|^{\frac{\rho-s_k}{d}}} +
\frac{ |\langle f,\vf_{k_0}\rangle|\log(m)}{|H^+_m|^{1/2}} $$
where the last term only appears if $s_{k_0}=\rho-d/2$ for some form $\vf_{k_0}$, and the $\log(m)$ in the first term can be omitted if $d<2\rho$. \end{thm}
\begin{proof}
Write $f=\langle f, 1\rangle+\sum_k \langle f,\vf_{k}\rangle \vf_k +f_0$ with $f_0\in L^2_{\mathrm{temp}}(\G\bk G)$.
Then
\begin{equation}\label{e:spectral}\|\beta^+_mf-\mu(f)\|_2 \leq \sum_k |\langle f,\vf_k\rangle| \|\beta^+_m \vf_k\|_2+\|\beta^+_m f_0\|_2.
\end{equation}
We first estimate 
\begin{eqnarray*}
\|\beta^+_m f_0\|_2^2&=&\frac{1}{|H^+_m|^2}\sum_{h,h'\in H^+_m}\langle \pi(h')f_0,\pi(h)f_0\rangle\\
 &=&\frac{1}{|H^+_m|^2}\sum_{h\in H^+_m}\sum_{h'\in H_m^+}\langle \pi(h^{-1}h')f_0,f_0\rangle\\
 &=&\frac{1}{|H^+_m|^2}\sum_{h'\in H_m}\langle \pi(h')f_0,f_0\rangle\#\{h\in H^+_m: h'h\in H^+_m\}\\
 &\leq &\frac{1}{|H^+_m|}\sum_{h\in H_m}|\langle \pi(h)f_0,f_0\rangle|
\end{eqnarray*}
Noting that $f_0$ is spherical and tempered and using Proposition \ref{p:smatrixdecay} and Lemma \ref{l:kakunipotent} for $h=n_\bbk$ with $\|\bbk\|\geq 1$ gives
$$|\langle \pi(h)f_0,f_0\rangle|\ll \frac{\|f_0\|_2^2\log(\|\bbk\|)}{\|\bbk\|^{2\rho}},$$
and we can bound
\begin{eqnarray*}
\sum_{h\in H_m}|\langle \pi(h)f_0,f_0\rangle| &=& \sum_{\|\bbk\|<1}\norm{f_0}_2^2+
\sum_{1\leq \|\bbk\|\leq m}\frac{\|f_0\|_2^2\log(\|\bbk\|)}{\|\bbk\|^{2\rho}} \\
&\ll& \norm{f}_2^2\left\{\begin{array}{cc} 1 & d<2\rho\\ \log^2(m) & d=2\rho\\
\end{array}\right.
\end{eqnarray*}
so that 
$$\|\beta^+_m f_0\|_2\ll \frac{\|f\|_2}{|H_m^+|^{1/2}}\left\lbrace\begin{array}{cc}1& d<2\rho\\
  \log(m) & d=2\rho\end{array}\right.$$
  
By the same argument for each of the exceptional forms $\vf_{k}$ in $\pi_{s_k}$ we can bound 
\begin{eqnarray*} \|\beta^+_m \vf_k\|_2^2&\ll &\frac{1}{|H^+_m|}(1+\sum_{1\leq \|\bbk\|\leq {m}}\|\bbk\|^{2(s_k-\rho)})
\ll \left\{\begin{array}{cc} 
m^{-d} & s_k<\frac{2\rho-d}{2}\\
\log(m)m^{-d} & s_k=\frac{2\rho-d}{2}\\
m^{2(s_k-\rho)} & s_k>\frac{2\rho-d}{2}
\end{array}\right.
\end{eqnarray*}
hence
$$ \|\beta^+_m \vf_k\|_2\ll \left\{\begin{array}{cc} 
m^{-d/2} & s_k<\frac{2\rho-d}{2}\\
\log(m)m^{-d/2} & s_k=\frac{2\rho-d}{2}\\
m^{s_k-\rho} & s_k>\frac{2\rho-d}{2}
\end{array}\right..
$$
Plugging this in \eqref{e:spectral} and using the bound $|\langle f,\vf_k\rangle|\leq \|f\|_2$ for $s_k<\rho-d/2$ we get 
$$\|\beta^+_mf-\mu(f) \|_2 \ll \frac{\|f\|_2}{ |H_m^+|^{1/2}}+\sum_{s_k\geq \rho-d/2} |\langle f,\vf_k\rangle | m^{s_k-\rho} $$
where there is an extra $\log(m)$ multiplying the first term when $d=2\rho$ and multiplying $m^{\sigma_k-\rho}$ if $s_k=\rho-d/2$ for some $k$.
\end{proof}

When applying this to the shrinking target problem we take our test functions to be indicator functions of our shrinking sets. We now show that the result of Theorem \ref{t:SMET} can replace \eqref{e:MET} for studying spherical shrinking target problems. First, if our space is compact, or more generally if our shrinking sets are contained in a compact region $\Omega$ of our space, we can bound 
$|\langle f,\vf_k\rangle|\leq \|f\|_1\|{\vf_k}_{|_{\Omega}}\|_\infty\ll_\Omega  \|f\|_1$. We thus get

\begin{cor}\label{c:comp}
Fix a compact set $\Omega\subset \G\bk G$. For any spherical set $B\subseteq \Omega$ if $f=\chi_B$ then 
 for $d\leq 2\rho$
$$\|\beta^+_mf-\mu(f)\|_2 \ll_\Omega \frac{\sqrt{\mu(B)}\log(m)}{|H^+_m|^{1/2}}+ \frac{\mu(B)}{|H^+_m|^{\delta}},$$
where $\delta>0$ depends $d$ and on the spectral gap for $\G\bk G$, and the $\log(m)$ in the first term is only needed when $d=2\rho$.
\end{cor}

We are also interested in cases where our shrinking sets go far out into the cusps, in which case this result can not be applied. To handle these cases we recall that the exceptional forms $\vf_k$ are in $L^p(\G\bk G)$ for any $1<p<\frac{2\rho}{\rho-s_k}$. Taking $q=\frac{p}{p-1}$ we get that
$|\langle f,\vf_k\rangle|\ll  \|f\|_q,$
for all $q> \frac{2\rho}{\rho+s_k}$.  Using this bound together with Theorem \ref{t:SMET} gives

\begin{cor}\label{c:Lp}
For any spherical set $B\subseteq \G\bk G$, for $f=\chi_B$ we have for any $d\leq 2\rho$ 
$$\|\beta^+_mf-\mu(f)\|_2\ll_\epsilon \frac{\sqrt{\mu(B)}\log(m)}{|H_m^+|^{\frac{1}{2}}}+\sum_{s_k>\rho-\frac{d}{2}} \frac{ \mu(B)^{\frac{\rho+s_k-\epsilon}{2\rho}}}{|H_m^+|^{\frac{\rho-s_k}{d}}}+\frac{\log(m)\mu(B)^{1-\frac{d+2\epsilon}{4\rho}}}{|H_m^+|^{\frac12}},$$
where the $\log(m)$ in the first term is only needed when $d=2\rho$ and the last term only exists if $s_{k_0}=\rho-\tfrac{d}{2}$ for some residual form $\vf_{k_0}$.
\end{cor}

\subsection{Diagonalizable action}
Next we consider the case of diagonalizable actions. Here the $\Z$ action is given by a diagonalizable group $H=\{g_m\}_{m\in \Z}$. It is not hard to see that the argument we used for the unipotent action works just as well for diagonalizable actions. Moreover,  for diagonalizable actions the exponential decay of correlation can be used to obtain a similar mean ergodic theorem also for non-spherical smooth test functions. In this case, the bound on the right hand side will depend on Sobolev norms, but only logarithmically which is harmless for most applications.

\begin{thm}\label{t:DMET}
For $H=\{g_m\}_{m\in \Z}$ diagonalizable, for  $f\in L^2(\G\bk G)$ spherical we have
 $$\|\beta^+_mf-\mu(f)\|_2^2 \ll  \frac{\|f\|^2_2}{m},$$
 while for any smooth $f\in L^2(\G\bk G)\cap C^\infty(\G\bk G)$ we have, 
$$\|\beta^+_mf-\mu(f)\|_2^2 \ll  \frac{\|f\|^2_2\log(\tfrac{\cS(f)}{\|f\|_2})}{m},$$
where $\cS(f)=\cS_{l,2}(f)$ denotes the Sobolev norm from proposition \ref{p:matrixdecay} and
 the  implied constant depends on the spectral gap and the group $H$.
 \end{thm}
\begin{proof} 
Writing $f_0=f-\mu(f)$ and using the same argument as in the proof of Theorem \ref{t:SMET} we get
\begin{eqnarray*}
\|\beta^+_m(f_0)\|_2^2&=& \frac{1}{m^{2}}\sum_{0<k,k'\leq m}\langle \pi(g_k)f_0,\pi(g_{k'}f_0\rangle\\
&\leq& \frac{1}{m}\sum_{|k|\leq m}|\langle \pi(g_k)f_0,f_0\rangle| 
\end{eqnarray*}
First, when $f$ is spherical, so is $f_0$ in which case by Proposiiton \ref{p:smatrixdecay} and Lemma \ref{l:kakdiag} we can bound 
$$|\langle \pi(g_k)f_0,f_0\rangle|\ll \|f\|_2^2e^{-\alpha ck},$$
with $\alpha>0$ depending on the spectral gap and $c>0$ from Lemma \ref{l:kakdiag}. Plugging this bound back gives 
\begin{eqnarray*}
\|\beta^+_m(f_0)\|_2^2\ll  \frac{\|f\|_2^2}{m}\sum_{|k|\leq m}e^{-c\alpha k}\ll  \frac{\|f\|_2^2}{m},
\end{eqnarray*}
as claimed.

Next for a non-spherical smooth function, $f$, fix a parameter $1\leq M\leq m$ (to be determined later). For any $|k|\leq M$ bound $|\langle \pi(g_k)f_0,f_0\rangle|\leq \|f\|_2^2$ while for $|k|\geq M$, use Proposition \ref{p:matrixdecay} and Lemma \ref{l:kakdiag} to bound 
$$|\langle \pi(g_k)f_0,f_0\rangle| \ll \cS(f)^2e^{-\alpha c k},$$
to get
\begin{eqnarray*}
\|\beta_m^+(f_0)\|_2^2
&\ll& \frac{M}{m}\|f\|_2^2+\cS(f)^2\frac{1}{m}\sum_{M\leq |k|\leq m}e^{-\alpha c k}\\
&\ll& \frac{1}{m}(M\|f\|_2^2+  \cS(f)^2e^{-\alpha c M})
\end{eqnarray*}
Taking $M=\frac{2}{\alpha c}\log(\tfrac{\cS(f)}{\|f\|_2})$ we get that for all $m\geq 1$
\begin{eqnarray*}
\|\beta^+_m(f_0)\|_2^2
&\ll& \frac{\log(\tfrac{\cS(f)}{\|f\|_2})\|f\|_2^2}{m}\\
\end{eqnarray*}
as claimed.
\end{proof}

\section{Application to shrinking targets}
We now combine the general results from section \ref{s:Zdaction} and use the effective mean ergodic theorem to get  results about shrinking targets for unipotent $\Z^d$-actions and diagonalizable $\Z$-actions. 

We first set up our notation that will be fixed throughout this section. Here $H\leq G$ is either a discrete unipotent subgroup of rank $1\leq d<n$, that, without loss of generality we assume is given by \eqref{e:H} or a one parameter discrete diagonalizable group $H=\{g_k\}_{k\in \Z}$ (in which case $h_k=g_k$). In either case we let $H_m$ and $H_m^+$ denote the truncated orbit and forward orbit given in \eqref{e:Ht}, \eqref{e:Ht+} respectively. 

\subsection{Orbits hitting along a subsequence}
The following is a more general version of the second part of Theorem \ref{t:BCT} for multi-parameter unipotent actions.
\begin{prop}\label{p:multi1}
Let $H\cong \Z^d$ be unipotent and $\{B_m\}_{m\in \N}$ spherical shrinking targets.  For $d<2\rho$, if the sequence $\{m^d\mu(B_m)\}_{m\in \N}$ is unbounded then there is a subsequence $m_j$ such that $\lim_{j\to \infty}\frac{|xH^+_{m_j}\cap B_{m_j}|}{|H^+_{m_j}|\mu(B_{m_j})}= 1$ for a.e. $x\in \cX$ . For $d=2\rho$ this holds if we assume that $\{m^{d-\eta}\mu(B_m)\}_{m\in \N}$ is unbounded for some $\eta>0$ or if $B_1$ is precompact and $\frac{\mu(B_m)m^d}{(\log m)^2}$ is unbounded.
\end{prop}
\begin{proof}
For any $m\in \N$ let $f_m=\frac{\chi_{B_m}}{\mu(B_m)}$ and note that $\beta^+_m f_m(x)=\frac{|xH^+_m\cap B_m|}{|H^+_m|\mu(B_m)}$. When $d<2\rho$ by Corollary \ref{c:Lp} we have
$$\|\beta^+_mf_m-1\|\ll_\epsilon \frac{1}{\sqrt{\mu(B_m)m^d}}+ \sum_{s_k>\rho-\frac{d}{2}} \frac{1}{ \mu(B_m)^{\frac{\rho-s_k+\epsilon}{2\rho}}m^{\rho-s_k}}+\frac{ \log(m)}{\mu(B_m)^{\frac{d+2\epsilon}{4\rho}}m^{d/2}}.$$
For each $s_k$ we can rewrite
$$ \frac{1}{ \mu(B_m)^{\frac{\rho-s_k+\epsilon}{2\rho}}m^{\rho-s_k}}= \frac{1}{ [\mu(B_m)m^d]^{\frac{\rho-s_k+\epsilon}{2\rho}}m^{\rho-s_k-d\frac{\rho-s_k+\epsilon}{2\rho}}}.$$
Since $d<2\rho$ we can take  $\epsilon$ sufficiently small so that $\epsilon<\frac{2\rho-d}{4}$ and that $s_k\leq \rho(1- \frac{2d\epsilon}{\rho(2\rho-d)})$, implying that 
$\frac{\rho-s_k+\epsilon}{2\rho}>0$. Hence
$$\frac{1}{ \mu(B_m)^{\frac{\rho-s_k+\epsilon}{2\rho}}m^{\rho-s_k}}\leq \frac{1}{ [\mu(B_m)m^d]^{\frac{\rho-s_k+\epsilon}{2\rho}}},$$
and 
$$\frac{ \log(m)}{\mu(B_m)^{\frac{d+2\epsilon}{4\rho}}m^{d/2}}\leq \frac{1}{[\mu(B_m)m^d]^{\frac{d+2\epsilon}{4\rho}}}\frac{ \log(m)}{m^{\frac{2\rho-d}{4\rho}}}\ll \frac{1}{[\mu(B_m)m^d]^{\frac{d+2\epsilon}{4\rho}}}. $$
In particular, for a subsequence with  $m_j^d\mu(B_{m_j})\to \infty$ we have $\|\beta^+_{m_j}f_{m_j}-1\|\to 0$, and , after perhaps taking another subsequence, we get that  $\lim_{j\to \infty}\beta^+_{m_j}f_{m_j}(x)= 1$ for a.e. $x\in \cX$.

When $d=2\rho$ the main problem comes from the exceptional terms. In this case
$$ \frac{1}{ \mu(B_m)^{\frac{\rho-s_k+\epsilon}{2\rho}}m^{\rho-s_k}}= \frac{m^\epsilon}{ [\mu(B_m)m^d]^{\frac{\rho-s_k+\epsilon}{2\rho}}}.$$
In order to show that these terms go to zero we need to assume that $m^d\mu(B_m)\gg m^{\eta}$ for some $\eta>0$. This assumption will also take care of the extra $\log(m)$ terms and the proof follows as above.

Alternatively, if $\Omega=\overline{B_1}$ is compact,  we can use Corollary \ref{c:comp} to get
$$\|\beta^+_mf_m-1\|_2 \ll_\Omega \frac{\log(m)}{\sqrt{\mu(B)m^d}}+ \frac{1}{m^{d \delta}}.$$
Assuming that $\frac{\mu(B_m)m^d}{(\log m)^2}$ is unbounded we get that $\|\beta^+_{m_j}f_{m_j}-1\|_2\to 0$ for some subsequence and the result follows as before.
\end{proof}
When $H$ is diagonalizable the same argument using Theorem \ref{t:DMET} instead of Theorem \ref{t:SMET} gives
\begin{prop}\label{p:diag}
For $H=\{g_m\}_{m\in \Z}$ diagonalizable and $\{B_m\}_{m\in \N}$ spherical shrinking targets, if the sequence $\{m\mu(B_m)\}_{m\in \N}$ is unbounded then there is a subsequence $m_j$ such that $\lim_{j\to \infty}\frac{|xH^+_{m_j}\cap B_{m_j}|}{|H^+_{m_j}|\mu(B_{m_j})}= 1$ for a.e. $x\in \cX$ . 
\end{prop}

\subsection{Orbits eventually always hitting}
Next we consider general version of Theorem  \ref{t:asymp} for multi-parameter unipotent actions and one parameter diagonalizable actions.
Recalling Lemma \ref{l:CTBo} and Lemma \ref{l:CTB} we see that such a result will follow from effective estimates of the measures of $\cC_{m,\chi_{B_m}}$. For these we show
\begin{lem}\label{l:muCTB}
Let $B\subseteq \cX$ be spherical and let $f=\chi_B$. 
For $H\cong\Z^d$ unipotent with  $d<2\rho$ we have
$$\mu(\cC_{m,f})\ll_\epsilon \frac{1}{\mu(B)|H_m^+|}+
\sum_{s_k>\rho-\frac{d}{2}} \frac{1}{(\mu(B)|H_m^+|)^{1-\frac{s_k}{\rho}+\epsilon}|H_m^+|^{(\frac{2\rho}{d}-1)(1-\frac{s_k}{\rho})-\epsilon}}
$$
and for $d=2\rho$
$$\mu(\cC_{m,f})\ll_\epsilon \frac{\log^2(m)}{\mu(B)|H_m^+|}+\sum_{s_k} \frac{1}{(\mu(B)|H_m^+|)^{1-\frac{s_k}{\rho}+\epsilon}}
.$$
If $B$ is contained in some compact set $\Omega$ then 
$$\mu(\cC_{m,f})\ll_\Omega \frac{1}{|H_m^+|^\delta}+\left\lbrace\begin{array}{cc}\frac{1}{\mu(B)|H_m^+|} & d<2\rho
\\
&\\
 \frac{\log^2(m)}{\mu(B)|H_m^+|} & d=2\rho
\end{array}\right. .$$
\end{lem}
\begin{proof}
Since for any $x\in \cC_{m,f}$ we have that 
$|\beta^+_m(f)(x)-\mu(B)|\geq \frac{\mu(B)}{2}$, we can bound
$$\|\beta^+_m(f)-\mu(B)\|_2^2\geq \tfrac14\mu(\cC_{m,f})|\mu(B)|^2,$$
On the other hand by Corollary \ref{c:Lp} we have 
$$\|\beta^+_m(f)-\mu(B)\|_2^2\ll_\epsilon \sum_{s_k>\rho-\frac{d}{2}} \frac{\mu(B)^{\frac{\rho+s_k-\epsilon}{\rho}}}{|H_m^+|^{\frac{2(\rho-s_k)}{d}}}+\frac{\mu(B)\log^2(m)}{|H_m^+|}+\frac{\log^2(m)\mu(B)^{2-\epsilon-\frac{d}{2\rho}}}{|H_m^+|},$$
where the $\log(m)$ in the second to last term is only needed when $d=2\rho$, in which case the last term does not exist. In particular, when $d<2\rho$ the last term is dominated by the second to last term so that
\begin{equation}\label{e:muCTB}
\mu(\cC_{m,f})\ll_\epsilon \sum_{s_k>\rho-\frac{d}{2}} \frac{1}{ \mu(B)^{\frac{\rho-s_k+\epsilon}{\rho}}|H_m^+|^{\frac{2(\rho-s_k)}{d}}}+\frac{1}{\mu(B)|H_m^+|} .\end{equation}
and when $d=2\rho$
\begin{equation}\label{e:muCTBmax}
\mu(\cC_{m,f})\ll_\epsilon \sum_{s_k>\rho-\frac{d}{2}} \frac{1}{ \mu(B)^{\frac{\rho-s_k+\epsilon}{\rho}}|H_m^+|^{\frac{2(\rho-s_k)}{d}}}+\frac{\log^2(m)}{\mu(B)|H_m^+|}.\end{equation}
Finally, if we assume that $B$ is contained in a compact set $\Omega$, using Corollary \ref{c:comp} instead of Corollary \ref{c:Lp}  gives 
$$\mu(\cC_{m,f}) \ll_\Omega \frac{1}{|H_m^+|^{\delta}}+ \frac{\log^2(m)}{\mu(B_m)|H_m^+|},$$
where the $\log(m)$ term can be removed if $d<2\rho$.
\end{proof}

Using this estimate we get the following 
\begin{thm}\label{t:main}
Let $H\cong \Z^d$ be unipotent and $\{B_m\}_{m\in \N}$ spherical shrinking targets.
For $d<2\rho$, if 
\begin{equation}\label{e:summable3}
\sum_{j=0}^\infty \frac{1}{2^{dj}\mu(B_{2^j})}<\infty,
\end{equation}
then $\cA_{\rm ah}(\cB)$ is of full measure.
If we further assume that $\mu(B_{2^{j}})\ll \mu(B_{2^{j+1}})$ then for a.e. $x\in \cX$ for all $m$ sufficiently large
$\frac{|xH_m^+\cap B_m|}{|H_m^+|}\asymp \mu(B_m)$. 
For $d=2\rho$ the same holds if we replace the assumption \eqref{e:summable3} by the stronger assumption that $m^d \mu(B_m)\gg t^\eta$ for some $\eta>0$, or by the assumption that $B_1$ is pre-compact and 
$$\sum_{j=0}^\infty \frac{j^2}{2^{dj}\mu(B_{2^j})}<\infty.$$
\end{thm}

\begin{proof}
Let $f_m=\chi_{B_m}$ and let $b_m=m^d\mu(B_m)$.
From Lemma \ref{l:muCTB} we have that 
$$\mu(\cC_{2^{j\pm1},f_{2^j}})\ll_\epsilon \frac{1}{b_{2^j}}+
\sum_{s_k>\rho-\frac{d}{2}} \frac{1}{(b_{2^j})^{1-\frac{s_k}{\rho}+\epsilon}2^{jd[(\frac{2\rho}{d}-1)(1-\frac{s_k}{\rho})-\epsilon]}}
$$
The assumption that $\sum_j \frac{1}{b_{2^j}}<\infty$ implies that $b_{2^j}\to\infty$, so in particular $b_{2^j}\geq 1$ for $j$ sufficiently large. 
Taking $\epsilon$ sufficiently small so that $d[(\frac{2\rho}{d}-1)(1-\frac{s_k}{\rho})-\epsilon]\geq \delta>0$ for all $s_k$, we can  bound
$$\frac{1}{(b_{2^j})^{1-\frac{s_k}{\rho}+\epsilon}2^{jd[(\frac{2\rho}{d}-1)(1-\frac{s_k}{\rho})-\epsilon]}}\ll \frac{1}{2^{\delta j}}.$$
implying that 
$$\sum_{j}\mu(\cC_{2^{j\pm1},f_{2^{j}}})\ll \sum_j (\frac{1}{2^{\delta j}}+ \frac{1}{b_{2^j}})<\infty.$$
Since $\mu(\cC_{2^{j\pm 1},B_{2^{j}}}^o)\leq \mu(\cC_{2^{j\pm 1},f_{2^{j}}})$ the results follow from Lemma \ref{l:CTBo} and Lemma \ref{l:CTB} respectively.

Next, when $d=2\rho$ we have the weaker bound
$$\mu(\cC_{2^{j\pm 1},f_{2^j}})\ll_\epsilon \frac{j^2}{b_{2^j}}+\sum_{s_k} \frac{1}{(b_{2^j})^{1-\frac{s_k}{\rho}+\epsilon}}\ll \frac{j^2}{b_{2^j}}+\frac{1}{(b_{2^j})^\delta},$$
where $\delta>0$ is sufficiently small so that $1-\frac{s_k}{\rho}>\delta$ for all exceptional terms. In particular, if we assume that $b_m\gg m^\eta$ for some $\eta>0$ then $\sum_{j}\mu(\cC_{2^{j\pm 1},B_{2^{j}}})<\infty$ and the result follows as before.
Alternatively, if $B_1$ is pre-compact, we can use the bound
$$\mu(\cC_{2^{j\pm 1},f_{2^j}})\ll \frac{j^2}{b_{2^j}}+\frac{1}{2^{\delta j}}
$$
so $\mu(\cC_{2^{j\pm 1},f_{2^j}})$ is summable when $\frac{j^2}{b_{2^j}}$ is summable.
\end{proof}

When $H$ is diagonalizable we get a similar result also for non-spherical sets.
In this case we have the following Lemma.
\begin{lem}\label{l:muCmfD}
When $H$ is diagonalizable, for any spherical  $f\in L^2(\G\bk G)$ we have 
$$\mu(\cC_{m,f})\ll \frac{ \|f\|^2_2}{m|\mu(f)|^2}$$
and for smooth $f\in L^2(\G\bk G)$
$$\mu(\cC_{m,f})\ll \frac{\log(\frac{\cS(f)}{\|f\|_2}) \|f\|_2^2}{m|\mu(f)|^2}$$
\end{lem}
\begin{proof}
As before for any $x\in \cC_{m,f}$ we have that 
$|\beta^+_m(f)(x)-\mu(f)|\geq \frac{\mu(f)}{2}$ so we can bound
$$\|\beta^+_m(f)-\mu(f)\|_2^2\geq \tfrac14\mu(\cC_{m,B})|\mu(f)|^2.$$
Now use Theorem \ref{t:DMET} to bound
$$\|\beta^+_m(f)-\mu(f)\|_2^2\ll \frac{\|f\|_2^2}{m},$$
for spherical $f$ and 
$$\|\beta^+_m(f)-\mu(f)\|_2^2\ll \frac{\log(\tfrac{\cS(f)}{\|f\|_2})\|f\|_2^2}{m},$$
when $f$ is non spherical but smooth.
\end{proof}

\begin{prop}\label{p:diag1}
For $H=\{g_m\}_{m\in \N}$  diagonalizable and
$\cB=\{B_m\}_{m\in \N}\subseteq \G\bk G$ a family of shrinking targets. 
If the shrinking targets are spherical then  \eqref{e:summable} implies that 
$\mu(\cA_{\rm ah}(\cB))=1$ and further assuming  \eqref{e:double},
implies that for a.e. $x\in \cX$ for all $m$ sufficiently large
$\frac{|xH_m^+\cap B_m|}{|H_m^+|}\asymp \mu(B_m)$. 
For non-spherical regular shrinking targets, the same conclusion holds after replacing \eqref{e:summable} by  \eqref{e:summable2}.
\end{prop}
\begin{proof}
For spherical shrinking targets using Lemma \ref{l:muCmfD} for $f_m=\chi_{B_m}$ we get that 
$\mu(\cC^o_{m',f_m})\leq \mu(\cC_{m',f_m})\ll \frac{1}{m'\mu(B_m)},$
and the proof follows as in the unipotent case.

Next, assuming the sets are regular let  $0\leq f_m^-\leq \chi_{B_m}\leq f_m^+\leq 1$ approximate $B_m$ with  $\mu(B_m)\asymp \mu(f_m^\pm)$ and $\cS(f_m^\pm)\ll \mu(B_m)^{-\alpha}$. Applying Lemma  \ref{l:muCmfD} for $f_m^\pm$ we get that
$\mu(\cC_{m',f_m^\pm})\ll_\alpha \frac{|\log(\mu(B_{m})|}{ m'\mu(B_{m})}$ and since 
$\cC^o_{m', B_m}\subseteq \cC_{m',f_m^-}$ we also have that $\mu(\cC^o_{m', B_m})\ll\frac{|\log(\mu(B_{m})|}{ m'\mu(B_{m})}$. In particular,  \eqref{e:summable2} implies that $\sum_j\mu(\cC^o_{2^{j+1},B_{2^j}})<\infty$ and Lemma \ref{l:CTBo} implies that $\cA_{\rm ah}(\cB)$ is of full measure.

Finally, assuming  \eqref{e:double}  we get that 
$\mu(f^\pm_{2^j})\ll\mu(B_{2^j})\ll  \mu(B_{2^{j+1}})\ll  \mu(f^\pm_{2^{j+1}})$, 
and using Lemma \ref{l:CTB} we get that \eqref{e:summable2} implies $\beta^+_m(f_m^\pm)(x)\asymp \mu(f_m^\pm)\asymp \mu(B_m)$ for a.e. $x$. Since for all $x\in \cX$
$$\beta^+_m(f_m^-)(x)\leq\frac{|xH_m^+\cap B_m|}{|H_m^+|}\leq \beta^+_m(f_m^+)(x),$$
this concludes the proof.
 \end{proof}

\begin{rem}
We note that in order to show that $\cA_{\rm ah}(\cB)$ is of full measure we only used the approximation $0\leq f^-\leq \chi_B$.
It is thus enough to assume that all sets are  $(c,\alpha)$-inner regular, where a set $B$ is $(c,\alpha)$-inner regular if there is a smooth function $f\in C^\infty(\G\bk G)$  satisfying that $0\leq f\leq \chi_B$ with $\mu(B)\leq c\mu(f)$ and  $\cS(f)\leq \mu(B)^{-\alpha}$. 
\end{rem}

\subsection{Quasi-Independence}
The method used above relying on the effective mean ergodic theorem gives very strong results, however, it has the shortcoming that it only applies if the sequence $\{m\mu(B_m)\}_{m\in \N}$ is unbounded. For sequences with $\{m\mu(B_m)\}_{m\in \N}$ bounded (and $\sum_{m}\mu(B_m)=\infty$) we need to take a different approach using the  more standard notion of  quasi-independence.

The main ingredient is the following result going back to Schmidt. Let $\cF=\{f_m\}_{m\in \N}$ denote a sequence of functions on the probability space $(\cX,\mu)$ taking values in $[0,1]$.
For $M\in \N$ let $E^\cF_M=\sum_{m\leq M} \mu(f_m)$ and $S^\cF_M(x)=\sum_{m\leq M} f_m(x)$.
We also let 
$$R^\cF_{m,m'}=\mu(f_mf_{m'})-\mu(f_m)\mu(f_{m'}).$$
\begin{prop}\label{p:SP}(\cite[Lemma 2.6]{KleinbockMargulis1999})\\
Assuming that for some constant $C>0$, for all $M,N\in \N$
\begin{equation}\label{e:QI}
|\sum_{m,m'=M}^NR^\cF_{m,m'}|\leq C\sum_{m=M}^N \mu(f_m),
\end{equation}
then for a.e. $x\in \cX$ for any $\epsilon>0$
$$S_M^\cF(x)=E_M^\cF+O_\epsilon\left(\sqrt{E_M^\cF}\log^{\tfrac32+\epsilon}(E^\cF_M)\right).$$
\end{prop}

In particular, given a one parameter group, $\{g_m\}_{m\in\Z}$, and a  sequence of spherical sets, $\{B_m\}_{m\in \N}$, consider the functions $f_m(x)=\chi_{B_m}(xg_m)$ so that 
$S_M^\cF(x)=\#\{m\leq M: xg_m\in B_m\}$, and hence the condition \eqref{e:QI} implies that $\{m: xg_m\in B_m\}$ is unbounded for a.e. $x$ whenever $\sum_m \mu(B_m)=\infty$ (and moreover \eqref{e:SBC} holds).

We will show that the condition \eqref{e:QI} holds for any sequence of spherical shrinking targets with $m\mu(B_m)$ bounded, thus handling  all the missing cases not covered by the effective mean ergodic theorem.
\begin{prop}\label{p:UQI}
Let $G=\Isom(\bH^n)$ with $n\geq 3$ and $\G\bk G$ a lattice.
Let $\{B_m\}_{m\in \N}$ denote a sequence of spherical shrinking targets in $\cX=\G\bk G$ and assume that $m\mu(B_m)$ is uniformly bounded. Let $\{g_m\}_{m\in \Z}$ be an unbounded one parameter group, let $\cF=\{f_m\}_{m\in \N}$ with  $f_m(x)=\chi_{B_m}(xg_m)$. Then there is some $C>0$ such that for all $N>M\geq 1$
$$\sum_{m,m'=M}^NR^\cF_{m,m'}\leq C\sum_{m=M}^N \mu(f_m).$$
\end{prop}

\begin{proof}
Using the spectral decomposition we can write 
$$\chi_{B_m}=\mu(B_m)+ \sum_k \<\chi_{B_m},\vf_k\>\vf_k+ f_m^0,$$
with $f_m^0\in L^2_{\rm temp}(\G\bk G)$, and hence,  for any $m,m'$ we have 
\begin{eqnarray*}
\<\pi(g_m)\chi_{B_m},\pi(g_{m'})\chi_{B_m'}\>&=& \<\pi(g_{m-m'})\chi_{B_m},\chi_{B_m'})\>\\
&=& \mu(B_m)\mu(B_{m'})\\
&&+\sum_k \<\chi_{B_m},\vf_k\>\overline{\<\chi_{B_{m'}},\vf_k\>} \<\pi(g_{m-m'})\vf_k,\vf_k\>\\
&&+\<\pi(g_{m-m'})f_m^0,f_{m'}^0\>
\end{eqnarray*}

We first consider the case of a unipotent group. In this case, after conjugating by some element of $K$, we may assume that
$g_m=n_{mx}$ for some fixed $x\in\R^{n-1}$.
Using decay of matrix coefficients from Proposition \ref{p:smatrixdecay} (together with Lemma \ref{l:kakunipotent}), we can bound for any $m,m'$ with $l=|m-m'|\geq 1$
\begin{eqnarray*}
|R_{m,m'}^\cF|&\ll& \sum_k \frac{|\<\chi_{B_m},\vf_k\>\<\chi_{B_{m'}},\vf_k\>|}{l^{2(\rho-s_k)}}+
\frac{\|f_m^0\|_2\|f_{m'}^0\|_2\log(l)}{l^{2\rho}}\\
&\ll_\epsilon &\sum_k \frac{\mu(B_m)^{\frac{\rho+s_k}{2\rho}-\epsilon}\mu(B_{m'})^{\frac{\rho+s_k}{2\rho}-\epsilon}}{l^{2(\rho-s_k)}}+
\frac{\sqrt{\mu(B_m)\mu(B_{m'})}\log(l)}{l^{2\rho}}\\
\end{eqnarray*}
where we used H\"older inequality to bound $|\<\chi_{B_m},\vf_k\>|\leq \|\vf_k\|_{p_k}\|\chi_{B_m}\|_{q_k}$, with $p_k=(\tfrac{\rho-s_k}{2\rho}+\epsilon)^{-1}$ and $q_k=(\tfrac{\rho+s_k}{2\rho}-\epsilon)^{-1}$.

Since $R_{m,m'}^\cF=R_{m',m}^\cF$ and $|R_{m,m}^\cF|\leq \mu(B_m)$ we have 
\begin{eqnarray*}
\sum_{m,m'=M}^N|R_{m,m'}^\cF|&=&\sum_{m=M}^N|R_{m,m}^\cF|+2\sum_{l=1}^{N-M}\sum_{m=M}^{N} |R_{m,m+l}^\cF| \\
&\leq&\sum_{m=M}^N\mu(B_m)+2\sum_{l=1}^{\infty}\sum_{m=M}^{N} |R^\cF_{m,m+l}|
\end{eqnarray*}
To bound the second sum for each of the exceptional $s_k\in(0,\rho)$ we can bound
$$\mu(B_m)^{\frac{\rho+s_k}{2\rho}-\epsilon}\mu(B_{m'})^{\frac{\rho+s_k}{2\rho}-\epsilon}\leq \mu(B_m)\mu(B_l)^{s_k/\rho-2\epsilon},$$
where we used that $\mu(B_{m+l})\leq \mu(B_m)$ and also  $\mu(B_{m+l})\leq \mu(B_l)$, and similarly bound $\sqrt{\mu(B_m)\mu(B_{m+l})}\leq \mu(B_m)$ to get that 
\begin{eqnarray*}
\sum_{m,m'=M}^N|R_{m,m'}^\cF|\ll_\epsilon
\left(1+ \sum_{l=1}^\infty\frac{\log(l)}{l^{2\rho}}+ \sum_{0<s_k<\rho} \sum_{l=1}^\infty\frac{\mu(B_l)^{s_k/\rho-2\epsilon}}{l^{2(\rho-s_k)}}\right)\sum_{m=M}^M \mu(B_m)
\end{eqnarray*}
Since $2\rho=n-1$ the series $\sum_{l=1}^\infty\frac{\log(l)}{l^{2\rho}}$ converges for $n\geq 3$. For the exceptional terms, let $\delta_k=\rho-s_k>0$ and let $\delta=\min_k \delta_k>0$ denote the spectral gap. Then, since $l\mu(B_l)$ is uniformly bounded, we can bound
\begin{eqnarray*}
\sum_{l=1}^\infty\frac{\mu(B_l)^{s_k/\rho-2\epsilon}}{l^{2(\rho-s_k)}}\ll \sum_{l=1}^\infty\frac{1}{l^{1+
\delta_k(2-\rho^{-1})-2\epsilon}}
\end{eqnarray*}
Now, since $\rho=\frac{n-1}{2}\geq 1$ we get that $\delta_k(2-\rho^{-1})\geq \delta_k\geq \delta$ and we can take $\epsilon>0$ sufficiently small so that $\delta-2\epsilon>\delta/2$ so the series
\begin{eqnarray*}
\sum_{l=1}^\infty\frac{\mu(B_l)^{s_k/\rho-2\epsilon}}{l^{2(\rho-s_k)}}\ll \sum_{l=1}^\infty\frac{1}{l^{1+\delta/2}}
\end{eqnarray*}
also converges implying that 
\begin{eqnarray*}
\sum_{m,m'=M}^N|R^\cF_{m,m'}|\ll
\sum_{m=M}^M \mu(B_m)
\end{eqnarray*}
where the implied constant is independent of $M,N$.

Next we consider the easier case of $\{g_m\}_{m\in\Z}$ diagonalizable. 
Using Proposition \ref{p:smatrixdecay} with Lemma \ref{l:kakdiag} for diagonalizable flows we get that for $|m-m'|=l\geq 1$
$$
|R^\cF_{m,m'}|\ll \sqrt{\mu(B_m)\mu(B_{m'})}e^{-c\alpha \ell}
$$
Hence, as before
\begin{eqnarray*}
\sum_{m,m'=M}^N|R^\cF_{m,m'}|&=&\sum_{m=M}^N|R^\cF_{m,m}|+2\sum_{l=1}^{N-M}\sum_{m=M}^{N} |R^\cF_{m,m+l}| \\
&\ll&\sum_{l=0}^{\infty}e^{-c\alpha l}\sum_{m=M}^{N} \sqrt{\mu(B_m)\mu(B_{m+l})}\\
&\ll&\sum_{m=M}^{N} \mu(B_m)\\
\end{eqnarray*}
\end{proof}

\begin{rem}
For the case of a diagonalizable flow, the proof works also for $n=2$, and the assumption that 
$\{m\mu(B_m)\}_{m\in \N}$ is bounded is not needed.
\end{rem}
\subsection{Proof of main theorems}
The proof of Theorems \ref{t:BCT}, \ref{t:asymp}, and \ref{t:diag} follow from the above results as follows.
\begin{proof}[Proof of Theorem \ref{t:BCT}]
Let $\{g_m\}_{m\in \N}$ denote an unbounded one parameter subgroup of $G=\Isom(\bH^n)$ with $n\geq 3$ and  $\{B_m\}_{m\in \N}$ a family of spherical shrinking targets. 

If $\sum_{m}\mu(B_m)<\infty$ then by the easy half of the Borel Cantelli lemma,  $\{m: xg_m\in B_m\}$ is finite for a.e. $x\in \cX$ so that indeed $\mu(\cA_h(\cB))=0$.

If $\sum_{m}\mu(B_m)=\infty$ and $\{m\mu(B_m)\}_{m\in\N}$ is uniformly bounded, then Proposition \ref{p:UQI} and \ref{p:SP} imply that for a.e. $x\in \cX$
$$\lim_{M\to\infty}\frac{\#\{0\leq m\leq M: xg_m\in B_m\}}{\sum_{m\leq M}\mu(B_m)}=1.$$
and in particular, for a.e. $x$ the set $\{m: xg_m\in B_m\}$ is unbounded so $\mu(\cA_h(\cB))=1$.

Finally, if the sequence $\{m\mu(B_m)\}_{m\in \N}$ is unbounded, use Proposition \ref{p:multi1} with 
$d=1$ (for $\{g_m\}_{m\in \Z}$ unipotent) and Proposition \ref{p:diag} (if it is diagonalizable)
to get that there is a subsequence $m_j\to\infty$ with $m_j\mu(B_{m_j})\to\infty$ such that for a.e. $x$
$$\lim_{j\to\infty}\frac{\#\{0\leq m\leq m_j: xg_m\in B_{m_j}\}}{m_j\mu(B_{m_j})}=1.$$
In particular, since  
$$ \#\{0\leq m\leq m_j: xg_m\in B_{m}\}\geq \#\{0\leq m\leq m_j: xg_m\in B_{m_j}\}\gg m_j\mu(B_{m_j})$$ 
the set $\{m\geq 0: xg_m\in B_m\}$ is unbounded for a.e. $x\in \cX$, so again $\mu(\cA_h(\cB))=1$.
\end{proof}
 \begin{proof}[Proof of Theorem \ref{t:asymp}]
 Let $\cB=\{B_m\}_{m\in \N}$ denote a sequence of spherical shrinking targets and $\{g_m\}_{m\in \Z}< G$ an inbounded one parameter group, so $\{g_m\}_{m\in \Z}$ is either unipotent or diagonalizable . For the unipotent case the result follows from Theorem \ref{t:main} with $d=1$ and for the diagonalizable case from the first part of Proposition \ref{p:diag}.
 \end{proof}
 \begin{proof}[Proof of Theorem \ref{t:diag}]
Follows from the second part of Proposition \ref{p:diag}.
 \end{proof}

 \section{Logarithm laws}
 We now apply our general results on shrinking targets to prove logarithm laws for penetration depth. 
 \begin{proof}[Proof of Theorem \ref{t:strongloglaw}]
 Let $\{g_m\}_{m\in \N}$ denote an unbounded one parameter group (then it is either unipotent or diagonalizable).
 Fix $\epsilon>0$ and let $r_m^\pm=m^{-\frac{1\pm\epsilon}{n}}$
so that $\mu(B_{r_m^\pm}(x_0))=\frac{1}{m^{1\pm\epsilon}}$. 
 
 First, since $\sum_m \mu(B_{r_m^+}(x_0))<\infty$, then for a.e. $x\in \cX$ 
 the set $\{m: xg_m\in B_{r_m^+}(x_0)\}$ is bounded, so $d(xg_m,x_0)>r_m^+$ for all sufficiently large $m$. 
 Now, unless $x$ is in the orbit of $x_0$ (which is a null set) the condition that $d(xg_m,x_0)>r_m^+$ for all sufficiently large $m$ implies that also $d_m(x,x_0)>r_m^+$ for all sufficiently large $m$. Indeed, otherwise
 there is a sequence $m_j\to\infty$ and $k_j\leq m_j$ with $d(xg_{k_j},x_0)\leq r_{m_j}^+$, so $k_j\to\infty$  and $d(xg_{k_j},x_0)\leq r_{k_j}^+$, in contradiction. 
Now, the condition that $d_m(x,x_0)>r_m^+$ for all sufficiently large $m$ implies that $\frac{-\log(d_m(x,x_0))}{\log(m)}<\frac{1+\epsilon}{n}$ for all sufficiently large $m$ and hence 
 $$ \limsup_{m\to\infty}\frac{-\log(d_m(x,x_0))}{\log(m)}\leq \frac{1+2\epsilon}{n}.$$

Next, by Corollary  \ref{c:regular}  we get that for a.e. $x\in \cX$ for all sufficiently large $m$,
$\#\{x g_k: k\leq m\}\asymp m\mu( B_{r_m^-})=m^\epsilon$. Hence in particular, for a.e. $x\in \cX$ we have that  $d_m(x,x_0)<r_m^-$ for all sufficiently large $m$ implying that  
$$\liminf_{m\to\infty}\frac{-\log(d_m(x,x_0))}{\log(m)}\geq  \frac{1-2\epsilon}{n}.$$

We thus showed that for a.e. $x\in \cX$
 $$\frac{1-\epsilon}{n}\leq \liminf_{m\to\infty}\frac{-\log(d_m(x,x_0))}{\log(m)}\leq \limsup_{m\to\infty}\frac{-\log(d_m(x,x_0))}{\log(m)}\leq \frac{1+\epsilon}{n},$$
 and since this holds for any $\epsilon>0$ we have  $\lim_{m\to\infty}\frac{-\log(d_m(x,x_0))}{\log(m)}=\frac{1}{n}$ as claimed. 
 
The proof for the cusp neighborhoods are analogous where we use the cusp neighborhoods  $B_{r_m^\pm}(\infty)$ with $r_m^\pm=\frac{\log(m)(1\pm \epsilon)}{n-1}$ instead of the shrinking balls.
 \end{proof}
 
 \begin{rem}
 When $\{g_m\}_{m\in \N}$ is diagonalizable we can also consider  non $K$-invariant distance functions.
Since the corresponding norm balls and cusp neighborhoods are still regular we can use 
Theorem \ref{t:diag} instead of Theorem \ref{t:asymp} to get the same result.
 \end{rem}
 
\begin{proof}[Proof of corollary \ref{c:hitting}]
By  \cite[Proposition 11]{GalatoloPeterlongo10}) we have that 
$$\lim_{r\to 0} \frac{\log(\tau_r(x;x_0))}{-\log r}=(\lim_{m\to\infty}\frac{-\log(d_m(x,x_0))}{\log(m)})^{-1}$$
while an obvious modification of their argument gives
$$\lim_{r\to \infty} \frac{\log(\tau_r(x;\infty))}{r}=(\lim_{m\to\infty}\frac{(d_m(x,x_0))}{\log(m)})^{-1},$$
and the result follows immediately from Theorem \ref{t:strongloglaw}.
\end{proof}
 

%
%

\begin{thebibliography}{}

\bibitem[Ath13]{Athreya2013}
Jayadev~S. Athreya.
\newblock Cusp excursions on parameter spaces.
\newblock {\em J. Lond. Math. Soc. (2)}, 87(3):741--765, 2013.

\bibitem[AM09]{AthreyaMargulis09}
Jayadev~S. Athreya and Gregory~A. Margulis.
\newblock Logarithm laws for unipotent flows. {I}.
\newblock {\em J. Mod. Dyn.}, 3(3):359--378, 2009.


\bibitem[CK01]{ChernovKleinbock01}
N.~Chernov and D.~Kleinbock.
\newblock Dynamical {B}orel-{C}antelli lemmas for {G}ibbs measures.
\newblock {\em Israel J. Math.}, 122:1--27, 2001.

\bibitem[Dol04]{Dolgopyat04}
Dmitry Dolgopyat.
\newblock Limit theorems for partially hyperbolic systems.
\newblock {\em Trans. Amer. Math. Soc.}, 356(4):1637--1689 (electronic), 2004.

\bibitem[Fay06]{Fayad06}
Bassam Fayad.
\newblock Mixing in the absence of the shrinking target property.
\newblock {\em Bull. London Math. Soc.}, 38(5):829--838, 2006.

\bibitem[Gal07]{Galatolo07}
Stefano Galatolo.
\newblock Dimension and hitting time in rapidly mixing systems.
\newblock {\em Math. Res. Lett.}, 14(5):797--805, 2007.

\bibitem[GK15]{GhoshKelmer15}
A.~{Ghosh} and D.~{Kelmer}.
\newblock {Shrinking Targets for Semisimple Groups}.
\newblock Bull. Lond. Math. Soc. (to appear)


\bibitem[GP10]{GalatoloPeterlongo10}
Stefano Galatolo and Pietro Peterlongo.
\newblock Long hitting time, slow decay of correlations and arithmetical
  properties.
\newblock {\em Discrete Contin. Dyn. Syst.}, 27(1):185--204, 2010.

\bibitem[GS11]{GorodnikShah11}
Alexander Gorodnik and Nimish~A. Shah.
\newblock Khinchin's theorem for approximation by integral points on quadratic
  varieties.
\newblock {\em Math. Ann.}, 350(2):357--380, 2011.

\bibitem[HNPV13]{HaydnNicolPerssonVaienti13}
N.~Haydn, M.~Nicol, T.~Persson, and S.~Vaienti.
\newblock A note on {B}orel-{C}antelli lemmas for non-uniformly
              hyperbolic dynamical systems.
\newblock{\em Ergodic Theory Dynam. Systems} 33(2): 475--498, 2013.       


\bibitem[KM99]{KleinbockMargulis1999}
D.~Y. Kleinbock and G.~A. Margulis.
\newblock Logarithm laws for flows on homogeneous spaces.
\newblock {\em Invent. Math.}, 138(3):451--494, 1999.

\bibitem[KW16]{KleinbockWadleigh2016}
D.~Y. Kleinbock and N.~Wadleigh
\newblock A zero-one Law for improvements to Dirichlet's Theorem
\newblock {\em ArXiv e-prints}, September 2016.

\bibitem[KZ17]{KleinbockZhao2017}
D.~Y. Kleinbock and X.~Zhao.
\newblock An application of lattice points counting to shrinking target problems.
\newblock {\em ArXiv e-prints}, January 2017.


\bibitem[KM12]{KelmerMohammadi12}
D.~Kelmer and A.~Mohammadi.
\newblock Logarithm laws for one parameter unipotent flows.
\newblock {\em Geom. Funct. Anal.}, 22(3):756--784, 2012.

\bibitem[KY17]{KelmerYu17}
D.~Kelmer and S.~{Yu}.
\newblock Shrinking target problems for flows on homogenous spaces.
\newblock {In preperation}



\bibitem[Mau06]{Maucourant06}
Fran{\c{c}}ois Maucourant.
\newblock Dynamical {B}orel-{C}antelli lemma for hyperbolic spaces.
\newblock {\em Israel J. Math.}, 152:143--155, 2006.

\bibitem[MO15]{MohammadiOh15}
Amir Mohammadi and Hee Oh.
\newblock Matrix coefficients, counting and primes for orbits of geometrically
  finite groups.
\newblock {\em J. Eur. Math. Soc. (JEMS)}, 17(4):837--897, 2015.

\bibitem[Sha00]{Shalom2000}
Yehuda Shalom.
\newblock Rigidity, unitary representations of semisimple groups, and
  fundamental groups of manifolds with rank one transformation group.
\newblock {\em Ann. of Math. (2)}, 152(1):113--182, 2000.

\bibitem[Sul82]{Sullivan1982}
Dennis Sullivan.
\newblock Disjoint spheres, approximation by imaginary quadratic numbers, and
  the logarithm law for geodesics.
\newblock {\em Acta Math.}, 149(3-4):215--237, 1982.

\bibitem[{Yu}16]{Yu16}
S.~{Yu}.
\newblock {Logarithm laws for unipotent flows on $\Gamma\backslash
  \SO_0(n+1,1)$}.
\newblock {\em ArXiv e-prints}, October 2016.

\end{thebibliography}

\end{document}